\documentclass[a4paper,12pt,final]{amsart}
\usepackage{times,a4wide,mathrsfs,amssymb}
\usepackage{fdsymbol}

\newcommand{\C}{\mathbb{C}}
\newcommand{\ZZ}{\mathbb{Z}}

\newcommand{\QQ}{\mathbb{Q}}
\newcommand{\NN}{\mathbb{N}}
\newcommand{\PP}{\mathbb{P}}
\newcommand{\LL}{\mathbb{L}}
\newcommand{\GG}{\mathbb{G}}

\newcommand{\OO}{\mathcal O}

\newcommand{\Sy}{\mathfrak S}

\newcommand{\CC}{\mathcal C}

\newcommand{\MM}{\mathcal M}

\newcommand{\wt}{\widetilde}
\newcommand{\ima}{\hbox{Im}}
\newcommand{\rom}{\romannumeral}

    \makeatletter
\newcommand*{\da@rightarrow}{\mathchar"0\hexnumber@\symAMSa 4B }
\newcommand*{\da@leftarrow}{\mathchar"0\hexnumber@\symAMSa 4C }
\newcommand*{\xdashrightarrow}[2][]{%
  \mathrel{%
    \mathpalette{\da@xarrow{#1}{#2}{}\da@rightarrow{\,}{}}{}%
  }%
}
\newcommand{\xdashleftarrow}[2][]{%
  \mathrel{%
    \mathpalette{\da@xarrow{#1}{#2}\da@leftarrow{}{}{\,}}{}%
  }%
}
\newcommand*{\da@xarrow}[7]{%
  \sbox0{$\ifx#7\scriptstyle\scriptscriptstyle\else\scriptstyle\fi#5#1#6\m@th$}%
  \sbox2{$\ifx#7\scriptstyle\scriptscriptstyle\else\scriptstyle\fi#5#2#6\m@th$}%
  \sbox4{$#7\dabar@\m@th$}%
  \dimen@=\wd0 %
  \ifdim\wd2 >\dimen@
    \dimen@=\wd2 %
  \fi
  \count@=2 %
  \def\da@bars{\dabar@\dabar@}%
  \@whiledim\count@\wd4<\dimen@\do{%
    \advance\count@\@ne
    \expandafter\def\expandafter\da@bars\expandafter{%
      \da@bars
      \dabar@ 
    }%
  }%
  \mathrel{#3}%
  \mathrel{%
    \mathop{\da@bars}\limits
    \ifx\\#1\\%
    \else
      _{\copy0}%
    \fi
    \ifx\\#2\\%
    \else
      ^{\copy2}%
    \fi
  }%
  \mathrel{#4}%
}
\makeatother

\newtheorem{theorem}{Theorem}[section]
\newtheorem{claim}[theorem]{Claim}
\newtheorem{lemma}[theorem]{Lemma}
\newtheorem{corollary}[theorem]{Corollary}
\newtheorem{proposition}[theorem]{Proposition}
\newtheorem{conjecture}[theorem]{Conjecture}
\newtheorem{remark}[theorem]{Remark}
\newtheorem{definition}[theorem]{Definition}
\newtheorem{convention}{Conventions}

\newtheorem{problem}[theorem]{Problem}

\newtheorem{nonumbering}{Theorem}

\newtheorem{nonumberingt}{Acknowledgements}

\begin{document}
\author[Robert Laterveer]
{Robert Laterveer}

\address{Institut de Recherche Math\'ematique Avanc\'ee,
CNRS -- Universit\'e 
de Strasbourg,\
7 Rue Ren\'e Des\-car\-tes, 67084 Strasbourg CEDEX,
FRANCE.}
\email{robert.laterveer@math.unistra.fr}

\title{Algebraic cycles on a very special EPW sextic}

\begin{abstract} Motivated by the Beauville--Voisin conjecture about Chow rings of powers of $K3$ surfaces, we consider a similar conjecture for Chow rings of powers of EPW sextics. We prove part of this conjecture for the very special EPW sextic studied by Donten--Bury et alii. We also prove some other results concerning the Chow groups of this very special EPW sextic, and of certain related hyperk\"ahler fourfolds.
\end{abstract}



\maketitle

\section{Introduction}

For a smooth projective variety $X$ over $\C$, let $A^i(X)=CH^i(X)_{\QQ}$ denote the Chow group of codimension $i$ algebraic cycles modulo rational equivalence with $\QQ$--coefficients. Intersection product defines a ring structure on $A^\ast(X)=\oplus_i A^i(X)$. In the case of $K3$ surfaces, this ring structure has an interesting property:

\begin{theorem}[Beauville--Voisin \cite{BV}]\label{K3} Let $S$ be a $K3$ surface. Let $D_i, D_i^\prime\in A^1(S)$ be a finite number of divisors. Then
  \[ \sum_i D_i\cdot D_i^\prime=0\ \ \ \hbox{in}\ A^2(S)_{}\ \Leftrightarrow\ \sum_i D_i\cdot D_i^\prime=0\ \ \ \hbox{in}\ H^4(S,\QQ)\ .\]
  \end{theorem}

Conjecturally, a similar property holds for self--products of $K3$ surfaces:

\begin{conjecture}[Beauville--Voisin]\label{K3conj} Let $S$ be a $K3$ surface. For $r\ge 1$, let $D^\ast(S^r)\subset A^\ast(S^r)_{}$ be the $\QQ$--subalgebra generated by (the pullbacks of)
divisors and the diagonal of $S$. The restriction of the cycle class map induces an injection
  \[ D^i(S^r)\ \to\ H^{2i}(S^r,\QQ)\ \]
  for all $i$ and all $r$.
\end{conjecture}

(cf. \cite{V12}, \cite{V13}, \cite{Vo}, \cite{Y} for extensions and partial results concerning conjecture \ref{K3conj}.)

Beauville has asked which varieties have behaviour similar to theorem \ref{K3} and conjecture \ref{K3conj}. This is the problem of determining which varieties verify the ``weak splitting property'' of \cite{Beau3}. We briefly state this problem here as follows:

\begin{problem}[Beauville \cite{Beau3}]\label{prob} Find a nice class $\CC$ of varieties (containing $K3$ surfaces and abelian varieties), such that for any $X\in\CC$, the Chow ring of $X$ admits a multiplicative bigrading $A^\ast_{(\ast)}(X)$, with
  \[ A^i(X)=\bigoplus_{j\ge 0} A^i_{(j)}(X)\ \ \ \hbox{for\ all\ }i\ .\] 
This bigrading should split the conjectural Bloch--Beilinson filtration, in particular 
  \[ A^i_{hom}(X)= \bigoplus_{j\ge 1} A^i_{(j)}(X)\ .\]
  \end{problem}

It has been conjectured that hyperk\"ahler varieties are in $\CC$ \cite[Introduction]{Beau3}. Also, not all Calabi--Yau varieties can be in $\CC$ \cite[Example 1.7(b)]{Beau3}. An interesting novel approach of problem \ref{prob} (as well as a reinterpretation of theorem \ref{K3}) is provided by the concept of {\em multiplicative Chow--K\"unneth decomposition\/} (cf. \cite{SV}, \cite{V6}, \cite{SV2} and subsection \ref{ssmck} below).

In this note, we ask whether EPW sextics might be in $\CC$. An EPW sextic is a special sextic $X\subset\PP^5(\C)$ constructed in \cite{EPW}. EPW sextics are not smooth; however, a generic EPW sextic is a quotient $X=X_0/(\sigma_0)$, where $X_0$ is a smooth hyperk\"ahler variety (called a double EPW sextic) and $\sigma_0$ is an anti--symplectic involution \cite[Theorem 1.1]{OG}, \cite{OG3}.  Quotient varieties behave like smooth varieties with respect to intersection theory with rational coefficients, so the following conjecture makes sense:

\begin{conjecture}\label{optim} Let $X$ be an EPW sextic, and assume $X$ is a quotient variety $X=X_0/G$ with $X_0$ smooth and $G\subset\hbox{Aut}(X_0)$ a finite group. Then
$X\in\CC$.
\end{conjecture}

There are two reasons why conjecture \ref{optim} is likely to be true: first, because an EPW sextic is a Calabi--Yau hypersurface (and these are probably in $\CC$); secondly, because the 
hyperk\"ahler variety $X_0$ should be in $\CC$, and the involution $\sigma_0$ should behave nicely with respect to the bigrading on $A^\ast_{(\ast)}(X_0)$.
Let us optimistically suppose conjecture \ref{optim} is true, and see what consequences this entails for the Chow ring of EPW sextics. We recall that Chow groups are expected to satisfy a weak Lefschetz property, according to a long--standing conjecture:

\begin{conjecture}[Hartshorne \cite{Ha}]\label{weak} Let $X\subset\PP^{n+1}(\C)$ be a smooth hypersurface of dimension $n\ge 4$. Then the cycle class map
  \[ A^2_{}(X)_{}\ \to\ H^4(X,\QQ) \]
  is injective.
  \end{conjecture}

Conjecture \ref{weak} is notoriously open for all hypersurfaces of degree $d\ge n+2$.
Since quotient varieties behave in many ways like smooth varieties, it seems reasonable to expect that conjecture \ref{weak} extends to hypersurfaces that are quotient varieties. This would imply that an EPW sextic $X$ as in conjecture \ref{optim} has $A^2_{hom}(X)=0$. That is, conjecturally we have that
  \[ A^i(X)= A^i_{(0)}(X)\ \ \ \hbox{for\ all\ }i\le 2\ .\]
  For any $r\ge 1$, let us now define 
     \[ E^\ast(X^r)\ \subset \ A^\ast(X^r) \]
     as the $\QQ$--subalgebra generated by (pullbacks of) elements of $A^1(X)$ and $A^2(X)$ and the class of the diagonal of $X$. The above remarks imply a conjectural inclusion
     \[ E^\ast(X^r)\ \subset\ A^\ast_{(0)}(X^r)\ =\ A^\ast(X^r)/ A^\ast_{hom}(X^r) \ .\]  
     We thus arrive at the following concrete, falsifiable conjecture:
  
  \begin{conjecture}\label{subring} Let $X$ be an EPW sextic as in conjecture \ref{optim}.
   Then restriction of the cycle class map
     \[  E^i(X^r)\ \to\ H^{2i}(X^r,\QQ) \]
     is injective for all $i$ and all $r$.
     \end{conjecture}
     
  Conjecture \ref{subring} is the analogon of conjecture \ref{K3conj} for EPW sextics; the role of divisors on the $K3$ surface is played by (the hyperplane section and) codimension $2$ cycles on the sextic.
  The main result in this note provides some evidence for conjecture \ref{subring}: we can prove it is true for $0$--cycles and $1$--cycles on one very special EPW sextic:
  
  \begin{nonumbering}[=theorem \ref{main2}] Let $X$ be the very special EPW sextic of \cite{DBG}. Let $r\in\NN$.    
    The restriction of the cycle class map
    \[ E^i(X^r)\ \to\ H^{2i}(X^r,\QQ) \]
    is injective for $i\ge 4r-1$.
   \end{nonumbering}   

The very special EPW sextic of \cite{DBG} (cf. section \ref{secepw} below for a definition) is not smooth, but it is a ``Calabi--Yau variety with quotient singularities''. The very special EPW sextic $X$ is very symmetric; it is also
remarkable for providing the only example known so far of a complete family of $20$ pairwise incident planes in $\PP^5(\C)$ \cite{DBG}. 
As resumed in theorem \ref{epw} below, the very special EPW sextic $X$ is related to hyperk\"ahler varieties in two different ways: (a)
$X$ is rationally dominated via a degree $2$ map by the Hilbert scheme $S^{[2]}$ where $S$ is a $K3$ surface of Picard number $20$;
(b) $X$ admits a double cover that is the quotient of an abelian variety by a finite group of group automorphisms, and this quotient admits a hyperk\"ahler resolution $X_0$.

To prove theorem \ref{main2}, we first prove (proposition \ref{mck}) that the very special EPW sextic $X$ has a multiplicative Chow--K\"unneth decomposition, in the sense of Shen--Vial \cite{SV}, and so the Chow ring of $X$ has a bigrading. Next, 
we establish (proposition \ref{2}) that 
  \begin{equation}\label{0} A^2(X)=A^2_{(0)}(X)\ .\end{equation}
  Both these facts are proven using description (b), via the theory of {\em symmetrically distinguished cycles\/} \cite{OS}.
     
  Note that equality (\ref{0}) might be considered as evidence for conjecture \ref{weak} for $X$. In order to prove conjecture \ref{weak} 
  for the very special EPW sextic $X$, it remains to prove that
   \[ A^2_{(0)}(X)\cap A^2_{hom}(X)\stackrel{??}{=}0\ .\]
   Likewise, in order to prove the full conjecture \ref{subring} for the very special EPW sextic $X$, it remains to prove that
   \[ A^i_{(0)}(X^r)\cap A^i_{hom}(X^r)\stackrel{??}{=}0\ \ \ \hbox{for\ all\ }i, r\ .\]   
   We are not able to prove these equalities outside of the range $i\ge 4r-1$; this is related to some of the open cases of Beauville's conjecture on Chow rings of abelian varieties (remarks \ref{BB} and \ref{BBk}). 

On the positive side, we establish a precise relation between the Chow ring of the very special EPW sextic $X$ and the Chow ring of the 
hyperk\"ahler fourfold $X_0$ mentioned in description (b) (theorem \ref{main3}). This relation provides an alternative description of the splitting of the Chow ring of $X_0$ coming from a multiplicative Chow--K\"unneth decomposition (corollary \ref{XX0}). In proving this relation, we exploit description (a); a key ingredient in the proof is a strong version of the generalized Hodge conjecture for $X$ and $X_0$ (proposition \ref{ghc}), which crucially relies on the fact that the $K3$ surface $S$ has maximal Picard number.
  
We also obtain some results concerning Bloch's conjecture (subsection \ref{ssb}), as well as a conjecture of Voisin (subsection \ref{ssv}), for the very special EPW sextic. The application to Bloch's conjecture relies on description (b) (via the theory of symmetrically distinguished cycles), but also on description (a) (via the surjectivity result proposition \ref{surj}).

We end this introduction with a challenge: can one prove theorem \ref{main2} for other (not very special) EPW sextics ?

\vskip0.6cm

\begin{convention} In this note, the word {\sl variety\/} will refer to a reduced irreducible scheme of finite type over $\C$. A {\sl subvariety\/} is a (possibly reducible) reduced subscheme which is equidimensional. 

{\bf All Chow groups will be with rational coefficients}: we denote by $A_jX$ the Chow group of $j$--dimensional cycles on $X$ with $\QQ$--coefficients; for $X$ smooth of dimension $n$ the notations $A_jX$ and $A^{n-j}X$ will be used interchangeably. 

The notations $A^j_{hom}(X)$, $A^j_{num}(X)$, $A^j_{AJ}(X)$ will be used to indicate the subgroups of homologically trivial, resp. numerically trivial, resp. Abel--Jacobi trivial cycles.
The contravariant category of Chow motives (i.e., pure motives with respect to rational equivalence as in \cite{Sc}, \cite{MNP}) will be denoted $\MM_{\rm rat}$.



We will write $H^j(X)$ 
and $H_j(X)$ 
to indicate singular cohomology $H^j(X,\QQ)$,
resp. Borel--Moore homology $H_j(X,\QQ)$.
\end{convention}

\section{Preliminary material}

\subsection{Quotient varieties}

\begin{definition} A {\em projective quotient variety\/} is a variety
  \[ X=Y/G\ ,\]
  where $Y$ is a smooth projective variety and $G\subset\hbox{Aut}(Y)$ is a finite group.
  \end{definition}
  
 \begin{proposition}[Fulton \cite{F}]\label{quot} Let $X$ be a projective quotient variety of dimension $n$. Let $A^\ast(X)$ denote the operational Chow cohomology ring. The natural map
   \[ A^i(X)\ \to\ A_{n-i}(X) \]
   is an isomorphism for all $i$.
   \end{proposition}
   
   \begin{proof} This is \cite[Example 17.4.10]{F}.
      \end{proof}

\begin{remark} It follows from proposition \ref{quot} that the formalism of correspondences goes through unchanged for projective quotient varieties (this is also noted in \cite[Example 16.1.13]{F}). We can thus consider motives $(X,p,0)\in\MM_{\rm rat}$, where $X$ is a projective quotient variety and $p\in A^n(X\times X)$ is a projector. For a projective quotient variety $X=Y/G$, one readily proves (using Manin's identity principle) that there is an isomorphism
  \[  h(X)\cong h(Y)^G:=(Y,\Delta^G_Y,0)\ \ \ \hbox{in}\ \MM_{\rm rat}\ ,\]
  where $\Delta^G_Y$ denotes the idempotent ${1\over \vert G\vert}{\sum_{g\in G}}\Gamma_g$.  
  \end{remark}

\subsection{Finite--dimensionality}

We refer to \cite{Kim}, \cite{An}, \cite{MNP}, \cite{Iv}, \cite{J4} for basics on the notion of finite--dimensional motive. 
An essential property of varieties with finite--dimensional motive is embodied by the nilpotence theorem:

\begin{theorem}[Kimura \cite{Kim}]\label{nilp} Let $X$ be a smooth projective variety of dimension $n$ with finite--dimensional motive. Let $\Gamma\in A^n(X\times X)_{}$ be a correspondence which is numerically trivial. Then there is $N\in\NN$ such that
     \[ \Gamma^{\circ N}=0\ \ \ \ \in A^n(X\times X)_{}\ .\]
\end{theorem}

 Actually, the nilpotence property (for all powers of $X$) could serve as an alternative definition of finite--dimensional motive, as shown by a result of Jannsen \cite[Corollary 3.9]{J4}.
   Conjecturally, all smooth projective varieties have finite--dimensional motive \cite{Kim}. We are still far from knowing this, but at least there are quite a few non--trivial examples:
 
\begin{remark} 
The following varieties have finite--dimensional motive: abelian varieties, varieties dominated by products of curves \cite{Kim}, $K3$ surfaces with Picard number $19$ or $20$ \cite{P}, surfaces not of general type with $p_g=0$ \cite[Theorem 2.11]{GP}, certain surfaces of general type with $p_g=0$ \cite{GP}, \cite{PW}, \cite{V8}, Hilbert schemes of surfaces known to have finite--dimensional motive \cite{CM}, generalized Kummer varieties \cite[Remark 2.9(\rom2)]{Xu}, \cite{FTV},
 threefolds with nef tangent bundle \cite{Iy}, \cite[Example 3.16]{V3}, fourfolds with nef tangent bundle \cite{Iy2}, log--homogeneous varieties in the sense of \cite{Br} (this follows from \cite[Theorem 4.4]{Iy2}), certain threefolds of general type \cite[Section 8]{V5}, varieties of dimension $\le 3$ rationally dominated by products of curves \cite[Example 3.15]{V3}, varieties $X$ with $A^i_{AJ}(X)_{}=0$ for all $i$ \cite[Theorem 4]{V2}, products of varieties with finite--dimensional motive \cite{Kim}.
\end{remark}

\begin{remark}
It is an embarassing fact that up till now, all examples of finite-dimensional motives happen to lie in the tensor subcategory generated by Chow motives of curves, i.e. they are ``motives of abelian type'' in the sense of \cite{V3}. On the other hand, there exist many motives that lie outside this subcategory, e.g. the motive of a very general quintic hypersurface in $\PP^3$ \cite[7.6]{D}.
\end{remark}

The notion of finite--dimensionality is easily extended to quotient varieties:

\begin{definition} Let $X=Y/G$ be a projective quotient variety. We say that $X$ has finite--dimensional motive if the motive
  \[ h(Y)^G:= (Y, \Delta^G_Y,0)\ \ \ \in \MM_{\rm rat}\]
  is finite--dimensional. (Here, $\Delta^G_Y$ denotes the idempotent ${1\over \vert G\vert}{\sum_{g\in G}}\Gamma_g \in A^n(Y\times Y)$.) 
   \end{definition}
  
  Clearly, if $Y$ has finite--dimensional motive then also $X=Y/G$ has finite--dimensional motive. The nilpotence theorem extends to this set--up:
  
  \begin{proposition}\label{quotientnilp} Let $X=Y/G$ be a projective quotient variety of dimension $n$, and assume $X$ has finite--dimensional motive. Let $\Gamma\in A^n_{num}(X\times X)$. Then there is   
    $N\in\NN$ such that
     \[ \Gamma^{\circ N}=0\ \ \ \ \in A^n(X\times X)_{}\ .\]
  \end{proposition}
  
  \begin{proof} Let $p\colon Y\to X$ denote the quotient morphism.
     We associate to $\Gamma$ a correspondence $\Gamma_Y\in A^n(Y\times Y)$ defined as
    \[ \Gamma_Y:= {}^t \Gamma_p\circ \Gamma\circ \Gamma_p\ \ \ \in A^n(Y\times Y)\ .\]
  By Lieberman's lemma \cite[Lemma 3.3]{V3}, there is equality
   \[ \Gamma_Y =(p\times p)^\ast \Gamma\ \ \ \hbox{in}\ A^n(Y\times Y)\ ,\]
   and so $\Gamma_Y$ is $G\times G$--invariant:
   \[ \Delta_Y^G\circ \Gamma_Y\circ \Delta_Y^G =\Gamma_Y\ \ \ \hbox{in}\ A^n(Y\times Y)\ .\]
   This implies that 
     \[\Gamma_Y\in \Delta_Y^G\circ A^n(Y\times Y)\circ\Delta_Y^G\ ,\] 
     and so
   \[ \Gamma_Y\in\hbox{End}_{\MM_{\rm rat}}\bigl(h(Y)^G\bigr)\ .\]
   Since clearly $\Gamma_Y$ is numerically trivial, and $h(Y)^G$ is finite--dimensional (by assumption), there exists $N\in\NN$ such that
    \[ (\Gamma_Y)^{\circ N} = {}^t \Gamma_p\circ \Gamma\circ\Gamma_p\circ{}^t \Gamma_p\circ \cdots \circ \Gamma_p=0\ \ \ \hbox{in}\ A^n(Y\times Y)\ .\]
    Using the relation $\Gamma_p\circ{}^t \Gamma_p=d\Delta_X$, this boils down to
    \[ d^{N-1}\ \  {}^t \Gamma_p\circ \Gamma^{\circ N}\circ \Gamma_p=0\ \ \ \hbox{in}\ A^n(Y\times Y)\ .\]
    From this, we deduce that also
    \[ \Gamma^{\circ N}= {1\over d^{N+1}} \Gamma_p\circ \Bigl( d^{N-1} \ \ {}^t \Gamma_p\circ \Gamma^{\circ N}\circ \Gamma_p\Bigr) \circ {}^t \Gamma_p=0\ \ \ \hbox{in}\ A^n(X\times X)\ .\]    
     \end{proof}

\subsection{MCK decomposition}
\label{ssmck}

\begin{definition}[Murre \cite{Mur}] Let $X$ be a projective quotient variety of dimension $n$. We say that $X$ has a {\em CK decomposition\/} if there exists a decomposition of the diagonal
   \[ \Delta_X= \Pi_0+ \Pi_1+\cdots +\Pi_{2n}\ \ \ \hbox{in}\ A^n(X\times X)\ ,\]
  such that the $\Pi_i$ are mutually orthogonal idempotents and $(\Pi_i)_\ast H^\ast(X)= H^i(X)$.
\end{definition}

\begin{remark} The existence of a CK decomposition for any smooth projective variety is part of Murre's conjectures \cite{Mur}, \cite{J2}. If a quotient variety $X$
has finite--dimensional motive, and the K\"unneth components are algebraic, then $X$ has a CK decomposition (this can be proven just as \cite{J2}, where this is stated for smooth $X$).
\end{remark}

\begin{definition}[Shen--Vial \cite{SV}] Let $X$ be a projective quotient variety of dimension $n$. Let $\Delta^X_{sm}\in A^{2n}(X\times X\times X)$ be the class of the small diagonal
  \[ \Delta^X_{sm}:=\bigl\{ (x,x,x)\ \vert\ x\in X\bigr\}\ \subset\ X\times X\times X\ .\]
  An  MCK decomposition of $X$ is a CK decomposition $\{\Pi_i\}$ of $X$ that is {\em multiplicative\/}, i.e. it satisfies
  \[ \Pi_k\circ \Delta^X_{sm}\circ (\Pi_i\times \Pi_j)=0\ \ \ \hbox{in}\ A^{2n}(X\times X\times X)\ \ \ \hbox{for\ all\ }i+j\not=k\ .\]
  
  (NB: the acronym ``MCK'' is shorthand for ``multiplicative Chow--K\"unneth''.)
  \end{definition}
  
  \begin{remark} The small diagonal (seen as a correspondence from $X\times X$ to $X$) induces the {\em multiplication morphism\/}
    \[ \Delta^X_{sm}\colon\ \  h(X)\otimes h(X)\ \to\ h(X)\ \ \ \hbox{in}\ \MM_{\rm rat}\ .\]
 Suppose $X$ has a CK decomposition
  \[ h(X)=\bigoplus_{i=0}^{2n} h^i(X)\ \ \ \hbox{in}\ \MM_{\rm rat}\ .\]
  By definition, this decomposition is multiplicative if for any $i,j$ the composition
  \[ h^i(X)\otimes h^j(X)\ \to\ h(X)\otimes h(X)\ \xrightarrow{\Delta^X_{sm}}\ h(X)\ \ \ \hbox{in}\ \MM_{\rm rat}\]
  factors through $h^{i+j}(X)$.
  
  The property of having an MCK decomposition is severely restrictive, and is closely related to Beauville's ``weak splitting property'' \cite{Beau3}. For more ample discussion, and examples of varieties with an MCK decomposition, we refer to \cite[Section 8]{SV} and also \cite{V6}, 
  \cite{SV2}, \cite{FTV}.
    \end{remark}

  \begin{lemma}\label{hk} Let $X, X^\prime$ be birational hyperk\"ahler varieties. Then $X$ has an MCK decomposition if and only if $X^\prime$ has one.
  \end{lemma}
  
  \begin{proof} This is noted in \cite[Introduction]{V6}; the idea is that Rie\ss's result \cite{Rie} implies that $X$ and $X^\prime$ have isomorphic Chow motives and the isomorphism is compatible with the multiplicative structure. 
   
More precisely: let $\gamma\colon X\dashrightarrow X^\prime$ be a birational map between hyperk\"ahler varieties of dimension $n$, and suppose $\{\Pi^X_i\}$ is an MCK decomposition for $X$. Let $\Delta^X_{sm}, \Delta^{X^\prime}_{sm}$ denote the small diagonal of $X$ resp. $X^\prime$. As explained in \cite[Section 6]{SV}, the argument of \cite{Rie} gives the equality
  \[ \Gamma_\gamma\circ \Delta^X_{sm}\circ {}^t \Gamma_{\gamma\times \gamma} = \Delta^{X^\prime}_{sm}\ \ \ \hbox{in}\ A^{2n}(X^\prime\times X^\prime\times X^\prime)\ .\]
  The prescription
   \[ \Pi^{X^\prime}_i:= \Gamma_\gamma\circ \pi^X_i\circ {}^t \Gamma_\gamma\ \ \ \in A^n(X^\prime\times X^\prime) \]
   defines a CK decomposition for $F^\prime$.
   (The $\Pi^{X^\prime}_i$ are orthogonal idempotents thanks to Rie\ss's result that $\Gamma_\gamma\circ {}^t \Gamma_\gamma=\Delta_{X^\prime}$ and ${}^t \Gamma_\gamma   \circ \Gamma_\gamma=\Delta_X$ \cite{Rie}.)
        
   To see this CK decomposition $\{\Pi^{X^\prime}_i\}$ is multiplicative, let us consider integers $i,j,k$ such that $i+j\not=k$. It follows from the above equalities that
   \[  \begin{split} \Pi^{X^\prime}_k\circ \Delta^{X^\prime}_{sm}\circ (\Pi^{X^\prime}_i\times \Pi^{X^\prime}_j) &=\Gamma_\gamma\circ \Pi^X_k\circ {}^t \Gamma_\gamma\circ \Gamma_\gamma\circ \Delta^X_{sm}\circ
    {}^t \Gamma_{\gamma\times\gamma}\circ \Gamma_{\gamma\times\gamma}\circ (\Pi^X_i\times \Pi^X_j)\circ {}^t \Gamma_\gamma\\
                 &=\Gamma_\gamma\circ \Pi^X_k\circ \Delta^X_{sm}\circ (\Pi^X_i\times \Pi^X_j)\circ {}^t \Gamma_\gamma\\
                 &=0\ \ \ \hbox{in}\ A^{2n} (X^\prime\times X^\prime)\ .\\
                 \end{split} \]
        (Here we have again used Rie\ss's result that $\Gamma_\gamma\circ {}^t \Gamma_\gamma=\Delta_{X^\prime}$ and ${}^t \Gamma_\gamma   \circ \Gamma_\gamma=\Delta_X$.)
          \end{proof}       

 \subsection{Niveau filtration}
 
 \begin{definition}[Coniveau filtration \cite{BO}]\label{con} Let $X$ be a quasi--projective variety. The {\em coniveau filtration\/} on cohomology and on homology is defined as
  \[\begin{split}   N^c H^i(X,\QQ)&= \sum \ima\bigl( H^i_Y(X,\QQ)\to H^i(X,\QQ)\bigr)\ ;\\
                           N^c H_i(X,\QQ)&=\sum \ima \bigl( H_i(Z,\QQ)\to H_i(X,\QQ)\bigr)\ ,\\
                           \end{split}\]
   where $Y$ runs over codimension $\ge c$ subvarieties of $X$, and $Z$ over dimension $\le i-c$ subvarieties.
 \end{definition}

Vial introduced the following variant of the coniveau filtration:

\begin{definition}[Niveau filtration \cite{V4}]\label{niv} Let $X$ be a smooth projective variety. The {\em niveau filtration} on homology is defined as
  \[ \wt{N}^j H_i(X)=\sum_{\Gamma\in A_{i-j}(Z\times X)_{}} \ima\bigl( H_{i-2j}(Z)\to H_i(X)\bigr)\ ,\]
  where the union runs over all smooth projective varieties $Z$ of dimension $i-2j$, and all correspondences $\Gamma\in A_{i-j}(Z\times X)_{}$.
  The niveau filtration on cohomology is defined as
  \[   \wt{N}^c H^iX:=   \wt{N}^{c-i+n} H_{2n-i}X\ .\]
  
\end{definition}

\begin{remark}\label{is}
The niveau filtration is included in the coniveau filtration:
  \[ \wt{N}^j H^i(X)\subset N^j H^i(X)\ .\] 
  These two filtrations are expected to coincide; indeed, Vial shows this is true if and only if the Lefschetz standard conjecture is true for all varieties \cite[Proposition 1.1]{V4}. 
  
  Using the truth of the Lefschetz standard conjecture in degree $\le 1$, it can be checked \cite[page 415 "Properties"]{V4} that the two filtrations coincide in a certain range:
  \[  \wt{N}^j H^i(X)= N^j H^iX\ \ \ \hbox{for\ all\ }j\ge {i-1\over 2} \ .\]
  \end{remark}

\subsection{Refined CK decomposition}

\begin{theorem}[Vial \cite{V4}]\label{pi_2} Let $X$ be a smooth projective variety of dimension $n\le 5$. Assume the Lefschetz standard conjecture $B(X)$ holds (in particular, the K\"unneth components $\pi_i\in H^{2n}(X\times X)$ are algebraic). Then there is a splitting into mutually orthogonal idempotents
  \[ \pi_i=\sum_j \pi_{i,j}\ \ \ \in H^{2n}(X\times X)\ ,\]
  such that
   \[ (\pi_{i,j})_\ast H^\ast(X) =gr^j_{\wt{N}} H^i(X)\  .\]
 (Here, the graded $  gr^j_{\wt{N}} H^i(X)$ can be identified with a Hodge substructure of $H^i(X)$ using the polarization.)  
In particular,  
   \[ \begin{split} &(\pi_{2,1})_\ast H^j(X) = H^{2}(X)\cap F^1\ ,\\
                          &(\pi_{2,0})_\ast H^j(X)= H^2_{tr}(X)\ .\\
                       \end{split}     \]
                      (Here $F^\ast$ denotes the Hodge filtration, and $H^2_{tr}(X)$ is the orthogonal complement to $H^2(X)\cap F^1$ under the pairing
        \[ \begin{split} H^2(X)\otimes H^2(X)\ &\to\ \QQ\ ,\\
                                a\otimes b\ &\mapsto\ a\cup h^{n-2}\cup b\ .)\\
                         \end{split} \]      
 \end{theorem}
 
 \begin{proof} This is \cite[Theorem 1]{V4}.
 \end{proof}

\begin{theorem}[Vial \cite{V4}]\label{Pi_2} Let $X$ be as in theorem \ref{pi_2}. Assume in addition $X$ has finite--dimensional motive. Then there exists a CK decomposition $\Pi_i\in A^n(X\times X)$, and a splitting into mutually orthogonal idempotents
  \[ \Pi_i=\sum_j \Pi_{i,j}\ \ \ \in A^n(X\times X)\ ,\]
  such that 
    \[  \Pi_{i,j}=\pi_{i,j}\ \ \ \hbox{in}\ H^{2n}(X\times X)\ ,\]
   and
   \[  (\Pi_{2i,i})_\ast  A^k(X)=0\ \ \ \hbox{for\ all\ }k\not= i\ . \]
 The motive $h_{i,0}(X)=(X,\Pi_{i,0},0)\in \MM_{\rm rat}$ is well--defined up to isomorphism. 
    \end{theorem}
    
  \begin{proof} This is \cite[Theorem 2]{V4}. The last statement follows from \cite[Proposition 1.8]{V4} combined with \cite[Theorem 7.7.3]{KMP}.
  \end{proof}

\begin{remark} In case $X$ is a surface with finite--dimensional motive, there is equality
   \[h_{2,0}(X)=t_2(X)\ \ \ \hbox{in}\ \MM_{\rm rat}\ ,\]
   where $t_2(X)$ is the ``transcendental part of the motive'' constructed for any surface (not necessarily with finite--dimensional motive) in \cite{KMP}.
\end{remark} 

\begin{lemma}\label{indecomp} Let $X$ be a smooth projective variety as in theorem \ref{Pi_2}, and assume 
  \[ \dim H^2(X,\OO_X)=1\ .\] 
  Then the motive 
  \[ h_{2,0}(X)\in \MM_{\rm rat} \]
  is {\em indecomposable\/}, i.e. any non--zero submotive $M\subset h_{2,0}(X)$ is equal to $h_{2,0}(X)$.
  \end{lemma}
  
  \begin{proof} (This kind of argument is well--known, cf. for instance \cite[Corollary 3.11]{V8} or \cite[Corollary 2.10]{Ped} where this is proven for $K3$ surfaces with finite--dimensional motive.)
The idea is that there are no non--zero Hodge substructures strictly contained in $H^2_{tr}(X)$. Since the motive $M\subset h_{2,0}(X)$ defines a Hodge substructure 
  \[ H^\ast(M)\ \subset\ H^2_{tr}(X)\ ,\]
  we must have $H^\ast(M)=H^2_{tr}(X)$ and thus an equality of homological motives
  \[  M=h_{2,0}(X)\ \ \ \hbox{in}\ \MM_{\rm hom}\ .\]
  Using finite--dimensionality of $X$, it follows there is an equality of Chow motives
  \[  M=h_{2,0}(X)\ \ \ \hbox{in}\ \MM_{\rm rat}\ .\]
 \end{proof}

\begin{lemma}\label{equiv} Let $X_1, X_2$ be two projective quotient varieties of dimension $4$. Assume $X_1, X_2$ have finite--dimensional motive, verify the Lefschetz standard conjecture and
  \[ N^1_H  H^4(X_j)= \wt{N}^1 H^4(X_j)\ \ \ \hbox{for\ }j=1,2\ ,\]
  where $N^\ast_H$ is the Hodge coniveau filtration.
  Let $\Gamma\in A^4(X_1\times X_2)$ and $\Psi\in A^4(X_2\times X_1)$. The following are equivalent:
  
  \noindent
  (\rom1) 
    \[  \Gamma_\ast\colon\ \ H^{0,4}(X_1)\ \to\ H^{0,4}(X_2)\]
    is an isomorphism, with inverse $\Psi_\ast$;
    
  \noindent
  (\rom2)
      \[  \Gamma_\ast\colon\ \ H^{4}_{tr}(X_1)\ \to\ H^{4}_{tr}(X_2)\]
    is an isomorphism, with inverse $\Psi_\ast$;

\noindent
(\rom3)
  \[ \Gamma\colon\ \ h_{4,0}(X_1)\ \to\ h_{4,0}(X_2)\ \ \ \hbox{in}\ \MM_{\rm rat}\]
  is an isomorphism, with inverse $\Psi$.
  \end{lemma}

  \begin{proof} Assume (\rom1), i.e.
    \[  \Psi_\ast \Gamma_\ast=\hbox{id}\colon\ \ H^{0,4}(X_1)\ \to\ H^{0,4}(X_1)\ .\]
    Using the hypothesis $N_H^1=\wt{N}^1$, this implies
    \[ \Psi_\ast \Gamma_\ast=\hbox{id}\colon\ \ H^4(X_1)/\wt{N}^1\ \to\ H^4(X_1)/\wt{N}^1\ ,\]
    and so
    \begin{equation}\label{hom}  \bigl(\Psi\circ \Gamma \circ\Pi^{X_1}_{4,0}\bigr){}_\ast  =(\Pi^{X_1}_{4,0})_\ast\colon\ \ H^\ast(X_1)\ \to\ H^\ast(X_1)\ .\end{equation}
    Considering the action on $H^4_{tr}(X_1)$, this implies
    \[ \Psi_\ast \Gamma_\ast=\hbox{id}\colon\ \ H^4_{tr}(X_1)\ \to\ H^4_{tr}(X_1)\ .\]
    Switching the roles of $X_1$ and $X_2$, one finds that likewise $\Gamma_\ast \Psi_\ast=\hbox{id}$ on $H^4_{tr}(X_2)$, and so the isomorphism of (\rom2) is proven.
    
    Next, we note that it formally follows from equality (\ref{hom}) that $\Psi$ is left--inverse to
    \[ \Gamma\colon\ \ h_{4,0}(X_1)\ \to\ h_{4,0}(X_2)\ \ \ \hbox{in}\ \MM_{\rm hom}\ .\]
    Switching roles of $X_1$ and $X_2$, one finds $\Psi$ is also right--inverse to $\Gamma$ and so
     \[ \Gamma\colon\ \ h_{4,0}(X_1)\ \to\ h_{4,0}(X_2)\ \ \ \hbox{in}\ \MM_{\rm hom}\ \]
     is an isomorphism, with inverse $\Psi$. By finite--dimensionality, the same holds in $\MM_{\rm rat}$, establishing (\rom3).
       \end{proof}
  
\begin{remark} The equality 
  \[N^1_H  H^4(X_j)= \wt{N}^1 H^4(X_j)\]
  in the hypothesis of lemma \ref{equiv} is the conjunction of the generalized Hodge conjecture $N^1_H=N^1$ and Vial's conjecture $N^1=\wt{N}^1$.
\end{remark}

\subsection{Symmetrically distinguished cycles on abelian varieties}

\begin{definition}[O'Sullivan \cite{OS}] Let $A$ be an abelian variety. Let $a\in A^\ast(A)$ be a cycle. For $m\ge 0$, let
  \[ V_m(a)\ \subset\ A^\ast(A^m) \]
  denote the $\QQ$--vector space generated by elements
  \[ p_\ast \Bigl(  (p_1)^\ast(a^{r_1})\cdot (p_2)^\ast(a^{r_2})\cdot\ldots\cdot (p_n)^\ast(a^{r_n})\Bigr)\ \ \ \in A^\ast(A^m) \ . \]
Here $n\le m$, and $r_j\in\NN$, and $p_i\colon A^n\to A$ denotes projection on the $i$--th factor, and $p\colon A^n\to A^m$ is a closed immersion with each component $A^n\to A$ being either a projection
or the composite of a projection with $[-1]\colon A\to A$.

The cycle $a\in A^\ast(A)$ is said to be {\em symmetrically distinguished\/} if for every $m\in\NN$ the composition
  \[ V_m(a)\ \subset\ A^\ast(A^m)\ \to\ A^\ast(A^m)/A^\ast_{hom}(A^m) \]
  is injective.
\end{definition}

\begin{theorem}[O'Sullivan \cite{OS}]\label{os} The symmetrically distinguished cycles form a $\QQ$--subalgebra $A^\ast_{sym}(A)\subset A^\ast(A)$, and the composition
  \[  A^\ast_{sym}(A)\ \subset\ A^\ast(A)\ \to\ A^\ast(A)/A^\ast_{hom}(A) \]
  is an isomorphism. Symmetrically distinguished cycles are stable under pushforward and pullback of homomorphisms of abelian varieties.
\end{theorem}

\begin{remark} For discussion and applications of the notion of symmetrically distinguished cycles, in addition to \cite{OS} we refer to \cite[Section 7]{SV}, \cite{V6}, \cite{Anc}, \cite{LFu2}.
\end{remark}

\begin{lemma}\label{sym} Let $A$ be an abelian variety of dimension $g$. 

\noindent
(\rom1)
There exists an MCK decomposition $\{ \Pi_i^A\}$ that is self--dual and consists of symmetrically distinguished cycles.

\noindent(\rom2) Assume $g\le 5$, and let  $\{ \Pi_i^A\}$ be as in (\rom1). There exists a further splitting
  \[ \Pi_2^A= \Pi_{2,0}^A +\Pi_{2,1}^A\ \ \ \hbox{in}\ A^g(A\times A)\ ,\]
  where the $\Pi_{2,i}^A$ are symmetrically distinguished and $\Pi^A_{2,i}=\pi^A_{2,i}$ in $H^{2g}(A\times A)$.
\end{lemma}

 \begin{proof}

 \noindent
 (\rom1) An explicit formula for  $\{ \Pi_i^A\}$ is given in \cite[Section 7 Formula (45)]{SV}.

 \noindent
 (\rom2) The point is that $\Pi_{2,1}^A$ is (by construction) a cycle of type
   \[  \sum_j C_j\times D_j\ \ \ \hbox{in}\ A^g(A\times A)\ ,\]
   where $D_j\subset A$ is a symmetric divisor and $C_j\subset A$ is a curve obtained by intersecting a symmetric divisor with hyperplanes. This implies $\Pi^A_{2,1}$ is symmetrically distinguished.
   By assumption, $\Pi^A_2$ is symmetrically distinguished and hence so is $\Pi^A_{2,0}$. 
 \end{proof}

\subsection{The very special EPW sextic}
\label{secepw}

This subsection introduces the main actor of this tale: the very symmetric EPW sextic discovered in \cite{DBG}.

\begin{definition}[\cite{Bea}] A {\em hyperk\"ahler variety\/} is a simply--connected smooth projective variety $X$ such that $H^0(X,\Omega^2_X)$ is spanned by a nowhere degenerate holomorphic $2$--form.
\end{definition}

\begin{theorem}[Donten--Bury et alii \cite{DBG}]\label{epw} Let $X\subset\PP^5(\C)$ be defined by the equation
  \[  \begin{split}  x_0^6+&x_1^6+x_2^6+x_3^6+x_4^6+x_5^6 + \bigl( x_0^4x_1^2 +x_0^4x_2^2+\cdots+ x_4^2x_5^4)\\
                              &+ (x_0^2x_1^2x_2^2+x_0^2x_1^2x_3^2+\cdots+x_3^2x_4^2x_5^2) + x_0x_1x_2x_3x_4x_5=0\ .
                              \end{split}\]
                              (Note that the parentheses are symmetric functions in the variables $x_0,\ldots,x_5$.)
    
    \noindent
    (\rom1)                          
      The hypersurface $X$ is an EPW sextic (in the sense of \cite{EPW}, \cite{OG}).    
      
     \noindent
     (\rom2)
      Let $S$ be the $K3$ surface obtained from a certain Del Pezzo surface in \cite{Vin}, and let $S^{[2]}$ denote the Hilbert scheme of $2$ points on $S$. Then there is a rational map (of degree $2$)
      \[ \phi\colon\ \ S^{[2]}\ \dashrightarrow\ X\ .\] 
      There exists a commutative diagram
      \[   \begin{array}[c]{ccccccc}
           S^{[2]} &\xdashrightarrow{\rm flops} & \overline{S^{[2]}} & \xrightarrow{} & X^\prime:=E^4/(G^\prime) &\xleftarrow{}& X_0\\
                & {\scriptstyle \phi}\sebkarrow\ \ \ \ \ \ \ & &\ \ \ \ \ \ \swarrow{\scriptstyle g}&&&\\
                 && X&&&&\\
                 \end{array}\]
       Here all horizontal arrows are birational maps. $E$ is an elliptic curve and $X^\prime:=E^4/(G^\prime)$ is a quotient variety, and $X_0$ is a hyperk\"ahler variety with $b_2(X_0)=23$ which is a symplectic resolution of $X^\prime$. The morphism $g$ is a double cover; $X$ is a projective quotient variety $X=E^4/G$ where $G=(G^\prime,i)$ with $i^2\in G^\prime$. The groups $G^\prime$ and 
       $G$ consist of automorphisms that are group homomorphisms.   
       
       \noindent
       (\rom3) $S^{[2]}$ and $X_0$ 
       have finite--dimensional motive and a multiplicative CK decomposition.
              \end{theorem}

\begin{proof}

\noindent
(\rom1) \cite[Proposition 2.6]{DBG}.

\noindent
(\rom2) This is a combination of \cite[Proposition 1.1]{DBG} and \cite[Sections 5 and 6]{DBG}. (Caveat: the group that we denote $G^\prime$ is written $G$ in \cite{DBG}.)

\noindent
(\rom3) Vinberg's $K3$ surface has Picard number $20$; as such, it is a Kummer surface and has finite--dimensional motive. This implies (using \cite{CM}) that $S^{[2]}$ has finite--dimensional motive. As birational hyperk\"ahler varieties have isomorphic Chow motives \cite{Rie}, $X_0$ has finite--dimensional motive. The Hilbert scheme $S^{[2]}$ of any $K3$ surface $S$ has an MCK decomposition \cite[Theorem 13.4]{SV}. As the isomorphism of \cite{Rie} is an isomorphism of algebras in the category of Chow motives, $X_0$ also has an MCK decomposition (lemma \ref{hk}).
\end{proof}

\begin{remark} The singular locus of the very special EPW sextic $X$ consists of $60$ planes. Among these 60 planes, there is a subset of 20 planes which form a complete family of pairwise incident planes in $\PP^5(\C)$ \cite{DBG}. This is the maximal number of elements in a complete family of pairwise incident planes, and this seems to be the only known example of a complete family of 20 pairwise incident planes.
\end{remark}

\begin{remark} The variety $X_0$ is not unique. In \cite[Section 6]{DBW}, it is shown there exist $81^{16}$ symplectic resolutions of $E^4/(G^\prime)$ (some of them non--projective).
One noteworthy consequence of theorem \ref{epw} is that the varieties $X_0$ are of $K3^{[2]}$ type (this was not a priori clear from \cite{DBW}).
\end{remark}

\begin{remark} For a {\em generic\/} EPW sextic $X$, there exists a hyperk\"ahler fourfold $X_0$ (called a ``double EPW sextic'') equipped with an anti--symplectic involution $\sigma_0$ such that $X=X_0/(\sigma_0)$ \cite[Theorem 1.1 (2)]{OG}. For the very special EPW sextic $X$, I don't know whether such $X_0$ exists. (For this, one would need to show that the Lagrangian subspace $A$ defining the very special EPW sextic is in the Zariski open $\LL\GG(\wedge^3 V)^0\subset\LL\GG(\wedge^3 V)$ defined in \cite[page 3]{OG}.)
\end{remark}

\section{Some intermediate steps}

\subsection{A strong version of the generalized Hodge conjecture}

For later use, we record here a proposition, stating that the very special EPW sextic, as well as some related varieties, satisfy the hypothesis of lemma \ref{equiv}:

\begin{proposition}\label{ghc} Let $X_0$ be any hyperk\"ahler variety as in theorem \ref{epw} (i.e., $X_0$ is a symplectic resolution of $E^4/(G^\prime)$). Then
  \[ N^1_H H^4(X_0)= \wt{N}^1 H^4(X_0)\ .\]
 (Here $N_H^\ast$ denotes the Hodge coniveau filtration and $\wt{N}^\ast$ denotes the niveau filtration (definition \ref{niv}).)
 
 The same holds for $X^\prime:=E^4/(G^\prime)$ and for the very special EPW sextic $X$: 
 \[ \begin{split} N^1_H H^4(X^\prime)&= \wt{N}^1 H^4(X^\prime)\ ,\\
                        N^1_H H^4(X)&= \wt{N}^1 H^4(X)\ .\\
                   \end{split}   \] 
 \end{proposition}
 
\begin{proof} The point is that Vinberg's $K3$ surface $S$ has Picard number $20$, and so the corresponding statement is easily proven for $S^{[2]}$:

\begin{lemma}\label{rho} Let $S$ be a smooth projective surface with $q=0$ and $p_g(S)=1$. Assume $S$ is $\rho$--maximal (i.e. $\dim H^2_{tr}(S)=2$). Then
   \[ N^1_H H^4(S^{[2]}) = \wt{N}^1 H^4(S^{[2]})\ .\]
\end{lemma}

\begin{proof} Let $\widetilde{S\times S}\to S\times S$ denote the blow--up of the diagonal. As is well--known,
there are isomorphisms of homological motives
 \[ \begin{split} 
       h(S^{[2]}) &\cong h(\wt{S\times S})^{\Sy_2}\ ,\\
       h(\wt{S\times S}) &\cong h(S\times S)\oplus h(S)(1)\ \ \ \hbox{in}\ \MM_{\rm hom}\ ,\\
       \end{split}\]
 where $\Sy_2$ denotes the symmetric group on $2$ elements acting by permutation. It follows there is a correspondence--induced injection
  \[ H^4(S^{[2]})\ \hookrightarrow\ H^4(S\times S)\oplus H^2(S)\ .\]
  It thus suffices to prove the statement for $S\times S$. Let us write
   \[ H^2(S)=N\oplus T:=NS(S)\oplus H^2_{tr}(S)\ .\]
   We have
   \[ \begin{split}  N^1_H H^4(S\times S)&= H^4(S\times S)\cap F^1\\
     &= H^0(S)\otimes H^4(S)\oplus H^4(S)\otimes H^0(S)\oplus N\otimes N\oplus N\otimes T\oplus T\otimes N\\
           &\ \ \ \ \ \ \ \ \ \  \ \          \oplus (T\otimes T)\cap F^1\ .\\
     \end{split}\]
   All but the last summand are obviously in $\wt{N}^1$. 
   As to the last summand, we have that
   \[ (T\otimes T)\cap F^1=(T\otimes T)\cap F^2 \ . \]
 Since the Hodge conjecture is true for $S\times S$ (indeed, $S$ is a Kummer surface and the Hodge conjecture is known for powers of abelian surfaces \cite[7.2.2]{Ab}, \cite[8.1(2)]{Ab2}),
 there is an inclusion
  \[ (T\otimes T)\cap F^2 \ \subset\ N^2 H^4(S\times S)= \wt{N}^2 H^4(S\times S)\ ,\]
  and so the lemma is proven.

    \end{proof}

Since birational hyperk\"ahler varieties have isomorphic cohomology rings \cite[Corollary 2.7]{Huy}, and the isomorphism (being given by a correspondence) respects Hodge structures, 
this proves the result for $X_0$. Since $X_0$ dominates $X^\prime$ and $X$, the result for $X^\prime$ and $X$ follows. Proposition \ref{ghc} is now proven.
\end{proof}

\subsection{MCK for quotients of abelian varieties}

\begin{proposition}\label{mck} Let $A$ be an abelian variety of dimension $n$, and let $G\subset\hbox{Aut}_{\ZZ}(A)$ be a finite group of automorphisms of $A$ that are group homomorphisms. The quotient
  \[ X=A/G \]
  has a self--dual MCK decomposition.
  \end{proposition}
  
  \begin{proof} A first step is to show there exists a self--dual CK decomposition for $X$ induced by a CK decomposition on $A$:

  \begin{claim}\label{sym1} Let $A$ and $X$ be as in proposition \ref{mck}, and let $p\colon A\to X$ denote the quotient morphism. Let $\{\Pi_i^A\}$ be a CK decomposition as in lemma \ref{sym}(\rom1). Then
  \[ \Pi_i^X:= {1\over d} \Gamma_p\circ \Pi_i^A\circ {}^t \Gamma_p\ \ \ \in A^n(X\times X)\ ,\ \ \ i=0,\ldots,2n \]
  defines a self--dual CK decomposition for $X$.
\end{claim}

To prove the claim, we remark that clearly the given $\Pi_i^X$ lift the K\"unneth components of $X$, and their sum is the diagonal of $X$. 
We will make use of the following property:

\begin{lemma}\label{comm} Let $A$ be an abelian variety of dimension $n$, and let $\{\Pi_i^A\}$ be an MCK decomposition as in lemma \ref{sym}(\rom1). For any $g\in\hbox{Aut}_{\ZZ}(A)$, we have
  \[ \Pi_i^A\circ \Gamma_g = \Gamma_g\circ \Pi_i^A\ \ \ \hbox{in}\ A^n(A\times A)\ .\]
  \end{lemma}
  
 \begin{proof} Because $g_\ast H^i(A)\subset H^i(A)$, we have a homological equivalence
   \[ \Pi_i^A\circ \Gamma_g - \Gamma_g\circ \Pi_i^A=0\ \ \ \hbox{in}\ H^{2n}(A\times A)\ .\] 
   But the left--hand side is a symmetrically distinguished cycle, and so it is rationally trivial.
 \end{proof}

To see that $\Pi_i^X$ is idempotent, we note that
  \[ \begin{split} \Pi_i^X\circ \Pi_i^X &= {1\over d^2} \Gamma_p\circ \Pi_i^A\circ {}^t \Gamma_p \circ \Gamma_p\circ \Pi_i^A\circ {}^t \Gamma_p\\
                                                     &= {1\over d} \Gamma_p\circ \Pi_i^A\circ \bigl(\sum_{g\in G} \Gamma_g\bigr)\circ \Pi_i^A\circ {}^t \Gamma_p\\
                                                     &= {1\over d} \Gamma_p\circ \Pi_i^A\circ \Pi_i^A\circ \bigl(\sum_{g\in G} \Gamma_g\bigr)\circ {}^t \Gamma_p\\
                                                     &= {1\over d} \Gamma_p\circ \Pi_i^A \circ \bigl(\sum_{g\in G} \Gamma_g\bigr)\circ {}^t \Gamma_p\\
                                                     &={1\over d} \Gamma_p\circ \Pi_i^A \circ  {}^t \Gamma_p \circ \Gamma_p \circ {}^t \Gamma_p\\
                                                     &={1\over d} \Gamma_p\circ \Pi_i^A \circ  {}^t \Gamma_p \circ d\Delta_X\\
                                                          &=  \Gamma_p\circ \Pi_i^A \circ  {}^t \Gamma_p = \Pi_i^X\ \ \ \hbox{in}\ A^n(X\times X)\ .\\
                                              \end{split}\]
                                              (Here, the third equality is an application of lemma \ref{comm}, and the fourth equality is because $\Pi_i^A$ is idempotent.)    
                  The fact that the $\Pi_i^X$ are mutually orthogonal is proven similarly; one needs to replace $\Pi_i^X\circ \Pi_i^X$ by $\Pi_i^X\circ \Pi_j^X$ in the above argument. This proves claim \ref{sym1}.

  Now, it only remains to see that the CK decomposition $\{\Pi_i^X\}$ of claim \ref{sym1} is multiplicative.
  
  \begin{claim}\label{mult} The CK decomposition $\{\Pi_i^X\}$ given by claim \ref{sym1} is an MCK decomposition. 
  \end{claim}
  
  To prove claim \ref{mult}, let us consider the composition
   \[ \Pi^X_k\circ \Delta^X_{sm}\circ (\Pi^X_i\times \Pi^X_j)\ \ \ \in\ A^n(X\times X)\ ,\]
   where we suppose $i+j\not= k$.
   There are equalities
   \[    \begin{split}    
      \Pi^X_k\circ \Delta^X_{sm}\circ (\Pi^X_i\times \Pi^X_j)  &= {1\over d^3}\  \Gamma_p\circ \Pi^A_k\circ {}^t \Gamma_p\circ \Delta^X_{sm}\circ \Gamma_{p\times p}\circ (\Pi^A_i\times \Pi^A_j)\circ {}^t \Gamma_{p\times p}\\
                          &={1\over d}\ \Gamma_p\circ \Pi^A_k\circ \Delta^G_A\circ \Delta^A_{sm}\circ (\Delta^G_A\times \Delta^G_A)\circ   (\Pi^A_i\times \Pi^A_j)\circ {}^t \Gamma_{p\times p}\\
                       & = {1\over d}\ \Gamma_p\circ \Delta^G_A\circ \Pi^A_k\circ \Delta^A_{sm}\circ   (\Pi^A_i\times \Pi^A_j)\circ (\Delta^G_A\times \Delta^G_A)\circ {}^t \Gamma_{p\times p}\\
                       & =0\ \ \ \hbox{in}\ A^{2n}(X\times X\times X)\ .\\
                       \end{split}\]
                 Here, the first equality is by definition of the $\Pi^X_i$, the second equality is lemma \ref{sm} below, the third equality follows from lemma \ref{comm}, and the fourth equality is the fact that $\{\Pi^A_i\}$ is an MCK decomposition for $A$ (lemma \ref{sym}).      
        
     \begin{lemma}\label{sm} There is equality
     \[ \begin{split} {}^t \Gamma_p\circ \Delta^X_{sm}\circ \Gamma_{p\times p}&={1\over d} (\sum_{g\in G} \Gamma_g)\circ \Delta^A_{sm}\circ \bigl((\sum_{g\in G} \Gamma_g)\times 
         (\sum_{g\in G}           \Gamma_g)\bigr) \\  
         &= d^2 \Delta^G_A\circ \Delta^A_{sm}\circ (\Delta^G_A\times \Delta^G_A)\ \ \ \hbox{in}\ A^{2n}(A\times A\times A)\ .\\
         \end{split}\]                    
     \end{lemma}
                                                
                                 \begin{proof} The second equality is just the definition of $\Delta^G_A$. As to the first equality, we first note that
                                 \[  \Delta^X_{sm}  ={1\over d}(p\times p\times p)_\ast (\Delta^A_{sm}) = {1\over d}\Gamma_p\circ \Delta^A_{sm}\circ {}^t \Gamma_{p\times p}\ \ \ \hbox{in}\ A^{2n}(X\times X\times X)\ .\]
                 This implies that
                 \[ {}^t \Gamma_p\circ \Delta^X_{sm}\circ \Gamma_{p\times p}= {1\over d}\ {}^t \Gamma_p\circ \Gamma_p\circ \Delta^A_{sm}\circ {}^t \Gamma_{p\times p}\circ \Gamma_{p\times p}\ .\]
                 But ${}^t \Gamma_p\circ \Gamma_p=\sum_{g\in G} \Gamma_g$, and thus
                 \[ {}^t \Gamma_p\circ \Delta^X_{sm}\circ \Gamma_{p\times p}={1\over d}  (\sum_{g\in G} \Gamma_g)\circ \Delta^A_{sm}\circ \bigl((\sum_{g\in G} \Gamma_g)\times 
         (\sum_{g\in G}           \Gamma_g)\bigr) \ \ \ \hbox{in}\ A^{2n}(A\times A\times A)\ ,\]
         as claimed. 
  \end{proof}
    This ends the proof of proposition \ref{mck}.
  \end{proof}
  
  In the set--up of proposition \ref{mck}, one can actually say more about certain pieces $A^i_{(j)}(X)$:
  
 \begin{proposition}\label{2} Let $X=A/G$ be as in proposition \ref{mck}. Assume $n=\dim X\le 5$ and $H^{2}(X,\OO_X)=0$. Assume also 
 there exists $X^\prime=A/(G^\prime)$ where
 $G=(G^\prime,i)$ with $i^2\in G^\prime$, and the action of $i$ on $H^2(X^\prime,\OO_{X^\prime})$ is minus the identity.
  Then any CK decomposition $\{\Pi_i\}$ of $X$
 verifies
    \[  \begin{split} &(\Pi_2)_\ast A^j(X)=0\ \ \ \hbox{for\ all\ }j\not=1\ ,\\
                             &(\Pi_6)_\ast A^j(X)=0\ \ \ \hbox{for\ all\ }j\not=3\ .
                             \end{split}   \]
   \end{proposition} 
 
 \begin{proof} It suffices to prove this for one particular CK decomposition, in view of the following lemma:
 
 \begin{lemma} Let $X=A/G$ be as in proposition \ref{mck}. Let $\Pi, \Pi^\prime\in A^n(X\times X)$ be idempotents, and assume $\Pi-\Pi^\prime=0$ in $H^{2n}(X\times X)$. Then
   \[  (\Pi)_\ast A^i(X)=0\ \Leftrightarrow\ (\Pi^\prime)_\ast A^i(X)=0\ .\]
   \end{lemma}
   
   \begin{proof} This follows from \cite[Lemma 1.14]{V4}. Alternatively, here is a direct proof.
   Let $p\colon A\to X$ denote the quotient morphism, and let $d:=\vert G\vert$. One defines
   \[  \begin{split}  \Pi_A&:=  {1\over d}\ \  {}^t \Gamma_p\circ \Pi\circ \Gamma_p\ \ \ \in A^n(A\times A)\ ,\\
                             \Pi^\prime_A&:=  {1\over d} {}^t \Gamma_p\circ \Pi^\prime\circ \Gamma_p\ \ \ \in A^n(A\times A)\ .\\ 
                            \end{split}\]
                 It is readily checked $\Pi_A, \Pi_A^\prime$ are idempotents, and they are homologically equivalent.
                 
        Let us assume $(\Pi)_\ast A^i(X)=0$ for a certain $i$. Then also
        \[ (\Pi_A)_\ast p^\ast A^i(X)= \bigl( {1\over d} {}^t \Gamma_p\circ \Pi\circ\Gamma_p\circ {}^t \Gamma_p\bigr){}_\ast A^i(X)=   \bigl( {} {}^t \Gamma_p\circ \Pi\bigr){}_\ast A^i(X)   =0\ .\]
     By finite--dimensionality of $A$, the difference $\Pi_A-\Pi^\prime_A\in A^n_{hom}(A\times A)$ is nilpotent, i.e. there exists $N\in\NN$ such that
     \[ \bigl( \Pi_A-\Pi^\prime_A\bigr)^{\circ N}=0\ \ \ \hbox{in}\ A^n(A\times A)\ .\]
     Upon developing, this implies
     \[  \Pi^\prime_A=(\Pi^\prime_A)^{\circ N}= Q_1+\cdots +Q_N\ \ \ \hbox{in}\ A^n(A\times A)\ ,\]
     where each $Q_j$ is a composition
      \[ Q_j= Q_j^1\circ Q_j^2\circ\cdots\circ Q_j^N\ ,\]
      with $Q_j^k\in \{ \Pi_A,\Pi^\prime_A\}$, and at least one $Q_j^k$ is $\Pi_A$. Since by assumption $   (\Pi_A)_\ast p^\ast A^i(X)=0$, it follows that 
      \[  (Q_j)_\ast= (\hbox{something})_\ast (\Pi_A)_\ast  \bigl((\Pi^\prime_A)^{\circ r}\bigr){}_\ast  =0\colon\ \       p^\ast A^i(X)\ \to\ p^\ast A^i(X)\ \ \ \hbox{for\ all\ }j\ .\] 
      But then also
      \[ (\Pi^\prime_A)_\ast  p^\ast A^i(X) =\bigl( Q_1+\cdots +Q_N\bigr){}_\ast   p^\ast A^i(X)=0\ .\]
       \end{proof}
 
 Now, let us take a projector for $A$ of the form
    \[\Pi^{A}_2=\Pi^{A}_{2,0}+\Pi^{A}_{2,1}\in A^n(A\times A)\ ,\]
    where $\Pi^{A}_{2,0}, \Pi^{A}_{2,1}$ are as in lemma \ref{sym}. 
    
\begin{lemma}\label{sym2} Let $A$ be an abelian variety of dimension $n\le 5$, and let $G\subset\hbox{Aut}_{\ZZ}(A)$ be a finite subgroup. Let $\Pi^A_{2,0}$ be as in lemma \ref{sym}. Then
  \[  \Pi^A_{2,0}\circ \Delta^G_A = \Delta^G_A\circ \Pi^A_{2,0}\ \ \ \in A^n(A\times A) \]
  is idempotent. (Here, as before, we write $\Delta^G_A:={1\over \vert G\vert}{\sum_{g\in G}} \Gamma_g\in A^n(A\times A)$.)
 \end{lemma}
 
 \begin{proof} For any $g\in G$, we have the commutativity
   \[ \Pi^A_{2,0}\circ \Gamma_g=\Gamma_g\circ \Pi^A_{2,0}\ \ \ \hbox{in}\ A^n(A\times A)\ ,\ \ \ \hbox{for\ all\ }g\in G\ , \]
   established in lemma \ref{sym}(\rom2). (Indeed, these cycles are symmetrically distinguished by lemma \ref{sym}(\rom2), and their difference is homologically trivial because an automorphism $g\in G$ respects the niveau filtration.)
   
   This commutativity clearly implies the equality
    \[  \Pi^A_{2,0}\circ \Delta^G_A = \Delta^G_A\circ \Pi^A_{2,0}\ \ \ \in A^n(A\times A)\ . \]
    To check that  $\Pi^A_{2,0}\circ \Delta^G_A $ is idempotent, we note that
    \[     \Pi^A_{2,0}\circ \Delta^G_A \circ  \Pi^A_{2,0}\circ \Delta^G_A =    \Pi^A_{2,0}\circ \Pi^A_{2,0}\circ \Delta^G_A\circ \Delta^G_A    =   \Pi^A_{2,0}\circ \Delta^G_A \ \ \ \hbox{in}\ A^n(A\times A)\ .\]        
     \end{proof}

 Let us write $G=G^\prime\times\{1,i\}$. Since by assumption, $i_\ast=-\hbox{id}$ on $H^{2,0}(X^\prime)$, we have equality
   \[  {1\over 2} \Bigl( \Pi_{2,0}^A\circ \Delta_A^{G^\prime} + \Pi_{2,0}^A\circ \Delta_A^{G^\prime}\circ \Gamma_i\Bigr)=0\ \ \ \hbox{in}\ H^{2n}(A\times A)\ .\]
 On the other hand, the left--hand side is equal to the idempotent $\Pi^A_{2,0}\circ \Delta^G_A$. By finite--dimensionality, it follows that
  \[    \Pi^A_{2,0}\circ \Delta^G_A=0\ \ \ \hbox{in}\ A^n(A\times A)\ .\]
  Using Poincar\'e duality, we also have $i_\ast=-\hbox{id}$ on $H^{2,4}(X^\prime)$, and so (defining $\Pi^A_{6,2}$ as the transpose of $\Pi^A_{2,0}$) there is also an equality
  \[  \Pi^A_{6,2}\circ \Delta^G_A= {1\over 2}  \Bigl( \Pi_{6,2}^A\circ \Delta_A^{G^\prime} + \Pi_{6,2}^A\circ \Delta_A^{G^\prime}\circ \Gamma_i\Bigr)=0\ \ \ \hbox{in}\ H^{2n}(A\times A)\ ,\]  
  and hence, by finite--dimensionality
  \[  \Pi^A_{6,2}\circ \Delta^G_A =0\ \ \   \hbox{in}\ A^n(A\times A)\ .\]
    Since $\Pi^A_{2,1}$ does not act on $A^j(A)$ for $j\not=1$ (theorem \ref{Pi_2}), we find in particular that
   \[ (\Pi^A_2)_\ast =0\colon\ \ \ A^j(A)^G\ \to\ A^j(A)^G\ \ \ \hbox{for\ all\ }j\not=1\ .\]
   Likewise, since $\Pi^A_{6,3}={}^t \Pi^A_{2,1}$ does not act on $A^j(A)$ for $j\not=3$ (theorem \ref{Pi_2}), we also find that
   \[ (\Pi^A_6)_\ast =0\colon\ \ \ A^j(A)^G\ \to\ A^j(A)^G\ \ \ \hbox{for\ all\ }j\not=3\ .\]
   
   We now consider the CK decomposition for $X$ defined as in lemma \ref{sym1}:
   \[ \Pi_i^X:= {1\over d} \Gamma_p\circ \Pi_i^A\circ {}^t \Gamma_p\ \ \ \in A^n(X\times X)\ .\]
   This CK decomposition has the required behaviour:
   \[ \begin{split} (\Pi_2^X)_\ast A^j(X) &=  \Bigl( {1\over d} \Gamma_p\circ \Pi_2^A\circ {}^t \Gamma_p \Bigr){}_\ast A^j(X)\\
           &= ({1\over d}\Gamma_p)_\ast (\Pi_2^A)_\ast p^\ast A^j(X)\\
           &= ({1\over d}\Gamma_p)_\ast (\Pi_2^A)_\ast  A^j(A)^G=0\ \ \ \hbox{for\ all\ }j\not=1\ ,\\
        \end{split}\]  
      and likewise
      \[ (\Pi^X_6)_\ast A^j(X)=0\ \ \ \hbox{for\ all\ }j\not=3\ .\]  
    This proves proposition \ref{2}.    
 \end{proof}
 
 For later use, we record here a corollary of the proof of proposition \ref{2}:
  
 \begin{corollary}\label{symref} Let $A$ be an abelian variety of dimension $n\le 5$, and let $\Pi^A_{2,0}, \Pi^A_{2,1}$ be as in lemma \ref{sym}(\rom2). Let $p\colon A\to X=A/G$ be a quotient variety with $G\subset\hbox{Aut}_{\ZZ}(A)$. The prescription
 \[  \Pi^X_{2,i}:= \Gamma_p\circ \Pi^A_{2,i}\circ {}^t \Gamma_p\ \ \ \hbox{in}\ A^n(X\times X)\]
  defines a decomposition in orthogonal idempotents
   \[ \Pi^X_2= \Pi^X_{2,0}+\Pi^X_{2,1}\ \ \ \hbox{in}\ A^n(X\times X)\ .\]
  The $\Pi^X_{2,i}$ verify the properties of the refined CK decomposition of theorem \ref{Pi_2}. 
 \end{corollary}
 
 \begin{proof} One needs to check the $\Pi^X_{2,i}$ are idempotent and orthogonal. This easily follows from the fact that the $\Pi^A_{2,i}$ commute with $\Gamma_g$ for $g\in G$ (lemma \ref{sym2}).
 \end{proof}

 \subsection{A surjectivity statement}
 
 \begin{proposition}\label{surj} Let $X_0$ be a hyperk\"ahler fourfold as in theorem \ref{epw}. Let $A^\ast_{(\ast)}(X_0)$ be the bigrading defined by the MCK decomposition. Then the intersection product map
  \[ A^2_{(2)}(X_0)\otimes A^2_{(2)}(X_0)\ \to\ A^4_{(4)}(X_0) \]
  is surjective.
  
  The same holds for $X^\prime:=E^4/(G^\prime)$ as in theorem \ref{epw}: $X^\prime$ has an MCK decomposition, and the intersection product map
   \[ A^2_{(2)}(X^\prime)\otimes A^2_{(2)}(X^\prime)\ \to\ A^4_{(4)}(X^\prime) \]
  is surjective.
  \end{proposition}
  
  \begin{proof} The result of Rie\ss\, \cite{Rie} implies there is an isomorphism of bigraded rings
    \[ A^\ast_{(\ast)}(S^{[2]})\ \xrightarrow{\cong}\ A^\ast_{(\ast)}(X_0)\ .\]
    For the Hilbert scheme of any $K3$ surface $S$, the intersection product map 
    \[  A^2_{(2)}(S^{[2]})\otimes A^2_{(2)}(S^{[2]})\ \to\ A^4_{(4)}(S^{[2]}) \]
  is known to be surjective \cite[Theorem 3]{SV}. This proves the first statement.
  
  For the second statement, the existence of an MCK decomposition for $X^\prime$ is a special case of proposition \ref{mck}. To prove the surjectivity statement for $X^\prime$,
  we note that $\phi\colon X_0\to X^\prime$ is a symplectic resolution and so there are isomorphisms
    \[    \phi^\ast\colon\ \ \ H^{p,0}(X^\prime)\ \xrightarrow{\cong}\ H^{p,0}(X_0)\ \ \ (p=2,4)\ .\]
   Using lemma \ref{equiv} (which is possible thanks to proposition \ref{ghc}), this implies there are isomorphisms
   \[  \phi^\ast\colon\ \ \ H^{p}_{tr}(X^\prime)\ \xrightarrow{\cong}\ H^{p}_{tr}(X_0)\ \ \ (p=2,4)\ .\]  
   This means there is an isomorphism of homological motives
   \[ {}^t \Gamma_\phi\colon\ \ \ h_{p,0}(X^\prime)\ \xrightarrow{\cong}\ h_{p,0}(X_0) \ \ \ \hbox{in}\ \MM_{\rm hom}    \ \ \ (p=2,4)\ .\]
   By finite--dimensionality, there are isomorphisms of Chow motives
    \[ {}^t \Gamma_\phi\colon\ \ \ h_{p,0}(X^\prime)\ \xrightarrow{\cong}\ h_{p,0}(X_0) \ \ \ \hbox{in}\ \MM_{\rm rat}    \ \ \ (p=2,4)\ .\]
    Taking Chow groups, this implies there are isomorphisms
    \begin{equation}\label{isiso} (\Pi^{X_0}_p \circ {}^t \Gamma_\phi\circ \Pi^{X^\prime}_p)_\ast \colon\ \ (\Pi^{X^\prime}_p)_\ast A^i(X^\prime)\ \to\ (\Pi^{X_0}_p)_\ast 
    A^i(X_0)\ \ \ (p=2,4)\ .\end{equation}
    Let us now consider the diagram
    \[ \begin{array}[c]{ccc}
       A^2_{(2)}(X_0)\otimes A^2_{(2)}(X_0)   & \to&  A^4_{(4)}(X_0)\\
       \uparrow && \uparrow\\
        A^2_{}(X_0)\otimes A^2_{}(X_0)   & \to&  A^4_{}(X_0)\\
        \uparrow && \uparrow\\
     A^2_{(2)}(X^\prime)\otimes A^2_{(2)}(X^\prime)   & \to&  A^4_{(4)}(X^\prime)\\
      \end{array}\]
      Here, the vertical arrows in the upper square are given by projecting to direct summand; the vertical arrows in the lower square are given by $\phi^\ast$. Since pullback and intersection product commute, the lower square commutes. Since $A^\ast_{(\ast)}(X_0)$ is a bigraded ring, the upper square commutes.
      
      The composition of vertical arrows is an isomorphism by (\ref{isiso}). The statement for $X^\prime$ now follows from the statement for $X_0$. 
        \end{proof}

 \section{Main results}

\subsection{Splitting of $A^\ast(X)$}

\begin{theorem}\label{main} Let $X$ be the very special EPW sextic of theorem \ref{epw}. The Chow ring of $X$ is a bigraded ring
  \[ A^\ast(X)= A^\ast_{(\ast)}(X)\ ,\]
  where
  \[ \begin{split} A^1(X)&=A^1_{(0)}(X)=\QQ\ ,\\
                        A^2(X)&=A^2_{(0)}(X)\ ,\\
                        A^3(X)&=A^3_{(0)}(X)\oplus A^3_{(2)}(X)= \QQ\oplus A^3_{hom}(X)\ ,\\
                        A^4(X)&=A^4_{(0)}(X)\oplus A^4_{(4)}(X)=\QQ\oplus A^4_{hom}(X)\ .\\
                        \end{split}\]
  \end{theorem}

  \begin{proof} It follows from theorem \ref{epw} that $X$ is a quotient variety $X=E^4/G$ with $G\subset\hbox{Aut}_{\ZZ}(A)$. Moreover, there is another quotient variety $X^\prime=E^4/(G^\prime)$ where $G=(G^\prime,i)$ and $i^2\in G^\prime$ and such that $i$ acts on $H^2(X^\prime,\OO_{X^\prime})$ as $-\hbox{id}$. Applying proposition \ref{mck}, it follows that $X$ has an MCK decomposition $\{\Pi_i^X\}$. Applying proposition \ref{2}, it follows that 
  \[  \begin{split}   &(\Pi^X_2)_\ast A^j(X)=0\ \ \  \hbox{for\ all\ } j\not=1\ ,\\
                      &(\Pi^X_6)_\ast A^j(X)=0\ \ \ \hbox{for\ all\ }j\not=3\ .\\
                     \end{split}\]
      The projectors $\Pi^X_i$ are $0$ for $i$ odd. (Indeed, $X$ has no odd cohomology so the $\Pi^X_i$ are homologically trivial. Using finite--dimensionality, they are rationally trivial.)                
  
  The projectors $\{\Pi_i^X\}$ define a multiplicative bigrading
  \[  A^\ast(X)= A^\ast_{(\ast)}(X)\ ,\]
  where $A^j_{(i)}(X):=(\Pi^X_{2j-i})_\ast A^j(X)$. The fact that $A^j_{(i)}(X)=0$ for $i<0$ follows from the corresponding property for abelian fourfolds \cite{Beau}.
  Likewise, the fact that
    \[ A^j_{(0)}(X)\cap A^j_{hom}(X)=0\ \ \ \hbox{for\ all\ }j\ge 3\]
    follows from the corresponding property for abelian fourfolds \cite{Beau}.
   \end{proof}

  \begin{corollary}\label{multip} Let $X$ be the very special EPW sextic. The intersection product maps
    \[ \begin{split}  A^2(X)\otimes A^2(X)\ &\to\ A^4(X)\ ,\\
                          A^2(X)\otimes A^1(X)\ &\to\ A^3(X)\ \\
                        \end{split}\]
               have image of dimension $1$.           
                            \end{corollary} 
   
 \begin{remark} It is instructive to note that for smooth Calabi--Yau hypersurfaces $X\subset\PP^{n+1}(\C)$, Voisin has proven that
  the intersection product map
    \[    A^{j}(X)\otimes A^{n-j}(X)\ \to\ A^n(X)  \]
    has image of dimension $1$, for any $0<j<n$ \cite[Theorem 3.4]{V13}, \cite[Theorem 5.25]{Vo} (cf. also \cite{LFu} for a generalization to generic complete intersections). 
    
    In particular, the first statement of corollary \ref{multip} holds for any smooth sextic in $\PP^5(\C)$. 
    The second statement of 
 corollary \ref{multip}, however, is not known (and maybe not true) for a general sextic in $\PP^5(\C)$. It might be that the second statement is specific to EPW sextics, and related to the presence of a hyperk\"ahler fourfold $X_0$ which is generically a double cover. 
  \end{remark}

 \begin{remark}\label{BB} Let $F^\ast$ be the filtration on $A^\ast(X)$ defined as
   \[ F^i A^j(X)=\bigoplus_{\ell\ge i} A^j_{(\ell)}(X)\ .\]
For this filtration to be of Bloch--Beilinson type, it remains to prove that
   \[ F^1 A^2(X) \stackrel{??}{=} A^2_{hom}(X) \ .\]
   This would imply the vanishing $A^2_{hom}(X)=0$ (i.e. the truth of conjecture \ref{weak} for $X$).
   
   Unfortunately, we cannot prove this. At least, it follows from the above description that the conjectural vanishing $A^2_{hom}(X)=0$ would follow from the truth of Beauville's conjecture
   \[ A^2_{hom}(E^4) \stackrel{??}{=} A^2_{(1)}(E^4)\oplus A^2_{(2)}(E^4)\ ,\]
  where $E$ is an elliptic curve.  
  \end{remark}

\subsection{Splitting of $A^\ast(X^r)$}

\begin{definition} Let $X$ be a projective quotient variety. For any $r\in\NN$, and any $1\le i<j<k\le r$, let
  \[ \begin{split} &p_j\colon\ \ X^r\ \to\ X\ ,\\
                   &p_{ij}\colon\ \ X^r\ \to\ X\times X\ ,\\
                   &p_{ijk}\colon\ \ X^r\ \to\ X\times X\times X\\
                   \end{split}\]
          denote projection on the $j$-th factor, resp. projection on the $i$-th and $j$-th factor, resp. projection on the $i$-th and $j$-th and $k$-th factor.
  
  We define        
   \[ E^\ast(X^r)\ \subset\ A^\ast(X^r) \]
   as the $\QQ$--subalgebra generated by 
    $(p_j)^\ast A^1(X)$ and $(p_j)^\ast A^2(X)$ and $(p_{ij})^\ast(\Delta_X)\in A^4(X^r)$ and $(p_{ijk})^\ast(\Delta^X_{sm})\in A^8(X^r)$.
\end{definition}

As explained in the introduction, the hypothesis that EPW sextics that are quotient varieties are in the class $\CC$ leads to the following concrete conjecture:

\begin{conjecture}\label{conjXk} Let $X\subset\PP^5(\C)$ be an EPW sextic which is a projective quotient variety. Let $r\in\NN$. The restriction of the cycle class map
  \[ E^i(X^r)\ \to\ H^{2i}(X^r) \]
  is injective for all $i$.
  \end{conjecture}
  
 For the very special EPW sextic, we can prove conjecture \ref{conjXk} for $0$--cycles and $1$--cycles:

 \begin{theorem}\label{main2} Let $X$ be the very special EPW sextic of definition \ref{epw}. Let $r\in\NN$.    
    The restriction of the cycle class map
    \[ E^i(X^r)\ \to\ H^{2i}(X^r) \]
    is injective for $i\ge 4r-1$.
   \end{theorem} 
    
   

  \begin{proof} The product $X^r$ has an MCK decomposition (since $X$ has one, and the property of having an MCK decomposition is stable under taking products \cite[Theorem 8.6]{SV}). Therefore, there is a bigrading on the Chow ring of $X^r$. As we have seen (theorem \ref{main}), 
 $A^1(X)=A^1_{(0)}(X)$ and $A^2(X)=A^2_{(0)}(X)$. Also, it is readily checked that 
   \[\Delta_X\in A^4_{(0)}(X\times X)\ .\] 
   (Indeed, this follows from the fact that
   \[ \Delta_X= \sum_{i=0}^8 \Pi_i^X = \sum_{i=0}^8 \Pi_i^X\circ \Delta_X\circ \Pi_i^X=\sum_{i=0}^8 (\Pi_i^X\times \Pi_{8-i}^X)_\ast \Delta_X\ \ \ \hbox{in}\ A^4(X\times X)\ ,\]
   where we have used the fact that the CK decomposition is self--dual.)
   The fact that $X$ has an MCK decomposition implies that 
     \[ \Delta^X_{sm}\ \ \in A^8_{(0)}(X\times X\times X)\] 
     \cite[Proposition 8.4]{SV}.
     
   Clearly, the pullbacks under the projections $p_i, p_{ij}, p_{ijk}$ respect the bigrading. (Indeed, suppose $a\in A^\ell_{(0)}(X)$, which means $a=(\Pi^X_{2\ell})_\ast (a)$. Then the pullback $(p_i)^\ast(a)$ can be written as 
     \[ X\times \cdots\times X\times (\Pi^X_{2\ell})_\ast (a)\times X\times\cdots\times X\ \ \ \in A^\ell(X^r)\ ,\]
     which is the same as
     \[  (\Pi^X_0\times\cdots \times\Pi^X_0\times\Pi^X_{2\ell}\times \Pi^X_0\times\cdots\times\Pi^X_0)_\ast (X\times \cdots\times X\times a\times X\times\cdots\times X)\ .\]
     This implies that 
       \[  (p_i)^\ast (a)\ \ \ \in (\Pi^{X^r}_{2\ell})_\ast A^\ell(X^r)=A^\ell_{(0)}(X^r)\ ,\]
       where $\Pi^{X^r}_\ast$ is the product CK decomposition. Another way to prove the fact that the projections  $p_i, p_{ij}, p_{ijk}$ respect the bigrading is by invoking  \cite[Corollary 1.6]{SV2}.)
            
   It follows there is an inclusion
     \[ E^\ast(X^r)\ \subset\ A^\ast_{(0)}(X^r)\ .\]
     The finite morphism $p^{\times r}\colon A^r\to X^r$ induces a split injection
     \[ (p^{\times r})^\ast\colon\ \ A^i_{(0)}(X^r)\cap A^i_{hom}(X^r)\ \to\ A^i_{(0)}(A^r)\cap A^i_{hom}(A^r)\ \ \ \hbox{for\ all\ }i.\]
     But the right--hand side is known to be $0$ for $i\ge 4r-1$ \cite{Beau}, and so
     \[  E^i(X^r)\cap   A^i_{hom}(X^r)\ \subset\       A^i_{(0)}(X^r)\cap A^i_{hom}(X^r)=0\ \ \ \hbox{for\ all\ }i\ge 4r-1\ .\] 
    \end{proof}

  \begin{remark}\label{BBk} As is clear from the proof of theorem \ref{main2}, there is a link with Beauville's conjectures for abelian varieties: let $E$ be an elliptic curve, and suppose one knows that
    \[  A^i_{(0)}(E^{4r})\cap A^i_{hom}(E^{4r})=0\ \ \ \hbox{for\ all\ }i\ \hbox{and\ all\ }r\ .\]
    Then conjecture \ref{conjXk} is true for the very special EPW sextic.
      \end{remark}

\subsection{Relation with some hyperk\"ahler fourfolds}
  
  \begin{theorem}\label{main3} Let $X$ be the very special EPW sextic of definition \ref{epw}. Let $X_0$ be one of the hyperk\"ahler fourfolds of \cite[Corollary 6.4]{DBW}, and let $f\colon X_0\to X$ be the generically $2:1$ morphism constructed in \cite{DBG}. Then $X_0$ has an MCK decomposition, and there is an isomorphism
    \[ \begin{split} f^\ast\colon\ \ A^4_{hom}(X)\ &\xrightarrow{\cong}\ A^4_{(4)}(X_0)\ .\\
                             \end{split} \]
   \end{theorem}

  \begin{proof} The MCK decomposition for $X_0$ was established in theorem \ref{epw}.
  The morphism $f\colon X_0\to X$ of \cite{DBG} is constructed as a composition
   \[ f\colon\ \ X_0\ \xrightarrow{\phi}\ X^\prime:=E^4/(G^\prime)\ \xrightarrow{ g}\ X\ ,\]
    where $\phi$ is a symplectic resolution and $g$ is the double cover associated to an anti--symplectic involution.
  This implies $f$ induces an isomorphism
   \[ f^\ast\colon\ \ H^{4,0}(X)\ \xrightarrow{\cong}\ H^{4,0}(X^\prime)\ \xrightarrow{\cong}\ H^{4,0}(X_0)\ .\]
   In view of the strong form of the generalized Hodge conjecture (proposition \ref{ghc}), $X_0$ and $X^\prime$ and $X$ verify the hypotheses of lemma \ref{equiv}. Applying lemma \ref{equiv}, we find isomorphisms of Chow motives
   \[ {}^t \Gamma_f\colon\ \ h_{4,0}(X)\ \xrightarrow{\cong}\ h_{4,0}(X^\prime) \ \xrightarrow{\cong}\ h_{4,0}(X_0)\ \ \ \hbox{in}\ \MM_{\rm rat}\ .\] 
   Since $(\Pi^X_{4,i})_\ast A^4(X)=0$ for $i\ge 1$ for dimension reasons, we have 
    \[  (\Pi^X_4)_\ast A^4(X)= (\Pi^X_{4,0})_\ast A^4(X) \ ,\]
    and the same goes for $X^\prime$ and $X_0$. It follows that
    \[ f^\ast\colon   A^4_{hom}(X)= A^4(  h_{4,0}(X))\ \xrightarrow{\cong}\ A^4(h_{4,0}(X_0))=:A^4_{(4)}(X_0)\ .\]    
     \end{proof}

  As a corollary, we obtain an alternative description of the splitting $A^\ast_{(\ast)}(X_0)$ for the hyperk\"ahler fourfolds $X_0$: 
   
    \begin{corollary}\label{XX0} Let $f\colon X_0\to X$ be as in theorem \ref{main3}. The splitting $A^\ast_{(\ast)}(X_0)$ (given by the 
   MCK decomposition of $X_0$) verifies
     \[  \begin{split}  A^4(X_0)&= A^4_{(4)}(X_0)\oplus A^4_{(2)}(X_0)\oplus A^4_{(0)}(X_0)\\
                       &= f^\ast A^4_{hom}(X)\oplus \ker\bigl( A^4(X_0)\xrightarrow{f_\ast} A^4(X)\bigr)\oplus \QQ\ ;\\
                                  A^3(X_0)&=A^3_{(2)}(X_0)\oplus A^3_{(0)}(X_0)\\
                                  &= A^3_{hom}(X_0)\oplus H^{3,3}(X_0)\ ;\\
                                  A^2(X_0)&= A^2_{(2)}(X_0)\oplus A^2_{(0)}(X_0)\\
                                  &=\ker\bigl( A^2_{hom}(X_0)\xrightarrow{f_\ast} A^2(X)\bigr) \oplus  A^2_{(0)}(X_0)\ .\\
                               \end{split}\]   
         \end{corollary}
     
\begin{remark} Just as we noted for the EPW sextic $X$ (remark \ref{BB}), for this filtration to be of Bloch--Beilinson type one would need to prove that
  \[ A^2_{(0)}(X_0) \cap A^2_{hom}(X_0)\stackrel{??}{=}0\ ,\]
  which I cannot prove. This situation is similar to that of the Fano varieties $F$ of lines on a very general cubic fourfold: thanks to work of Shen--Vial \cite{SV} there is a multiplicative bigrading
  $A^\ast_{(\ast)}(F)$ which has many good properties and interesting alternative descriptions. The main open problem is to prove that
    \[ A^2_{(0)}(F)\cap A^2_{hom}(F)\stackrel{??}{=}0\ ,\]
    which doesn't seem to be known for any single $F$.
\end{remark}


\begin{remark} Conjecturally, the relations of corollary \ref{XX0} should hold for any double EPW sextic $X_0$ (with $X$ being the quotient of $X_0$ under the anti--symplectic involution). However, short of knowing $X_0$ has finite--dimensional motive (as is the case here, thanks to the presence of the abelian variety $E^4$), this seems difficult to prove. 
Note that at least, for a general double EPW sextic $X_0$, the relations of corollary \ref{XX0} give a concrete description of a filtration on $A^\ast(X_0)$ that should be the Bloch--Beilinson filtration.
\end{remark}

\section{Further results}
\label{secf}

\subsection{Bloch conjecture}
\label{ssb} 

\begin{conjecture}[Bloch \cite{B}]\label{CB} Let $X$ be a smooth projective variety of dimension $n$. Let $\Gamma\in A^n(X\times X)$ be a correspondence such that
  \[ \Gamma_\ast=0\colon\ \ \ H^{p,0}(X)\ \to\ H^{p,0}(X)\ \ \ \hbox{for\ all\ }p>0\ .\]
  Then
  \[ \Gamma_\ast=0\colon\ \ \ A^n_{hom}(X)\ \to\ A^n_{hom}(X)\ .\]
\end{conjecture}

A weak version of conjecture \ref{CB} is true for the very special EPW sextic:

\begin{proposition} Let $X$ be the very special EPW sextic. Let $\Gamma\in A^4(X\times X)$ be a correspondence such that
  \[ \Gamma_\ast=0\colon\ \ \ H^{4,0}(X)\ \to\ H^{4,0}(X)\ .\]
  Then there exists $N\in\NN$ such that
   \[ (\Gamma^{\circ N})_\ast=0\colon\ \ \ A^4_{hom}(X)\ \to\ A^4_{hom}(X)\ .\]
  \end{proposition}
  
\begin{proof} As is well--known, this follows from the fact that $X$ has finite--dimensional motive; we include a proof for completeness' sake.

By assumption, we have
  \[    \Gamma_\ast=0\colon\ \ \ H^{4}(X,\C)/F^1\ \to\ H^{4}(X,\C)/F^1\ \]
  (where $F^\ast$ is the Hodge filtration). Thanks to the ``strong form of the generalized Hodge conjecture'' (proposition \ref{ghc}), this implies that also
  \[     \Gamma_\ast=0\colon\ \ \ H^{4}(X,\QQ)/\wt{N}^1\ \to\ H^{4}(X,\QQ)/\wt{N}^1\ .\]
  Using Vial's refined CK projectors (theorem \ref{Pi_2}), this means 
  \[ \Gamma\circ \Pi^X_{4,0}=0\ \ \ \hbox{in}\ H^8(X\times X)\ ,\]
  or, equivalently,
  \[ \Gamma - \sum_{(k,\ell)\not=(4,0)} \Gamma\circ \Pi^X_{k,\ell}=0 \ \ \ \hbox{in}\ H^8(X\times X)\ .\]  
  By finite--dimensionality, this implies there exists $N\in\NN$ such that
  \[ \Bigl(\Gamma - \sum_{(k,\ell)\not=(4,0)} \Gamma\circ \Pi^X_{k,\ell}\Bigr)^{\circ N}=0 \ \ \ \hbox{in}\ A^4(X\times X)\ .\]   
  Upon developing, this gives an equality
  \begin{equation}\label{devel} \Gamma^{\circ N}=Q_1+\cdots +Q_N\ \ \ \hbox{in}\ A^4(X\times X)\ ,\end{equation}     
  where each $Q_j$ is a composition of correspondences
   \[ Q_j=Q_j^1\circ Q_j^2\circ\cdots\circ Q_j^r\ \ \ \in A^4(X\times X)\ ,\]
   and for each $j$, at least one $Q_j^i$ is equal to $\Pi^X_{k,\ell}$ with $(k,\ell)\not=(4,0)$.
   Since (for dimension reasons)
    \[  (\Pi^X_{k,\ell})_\ast A^4_{hom}(X)=0\ \ \ \hbox{for\ all\ }(k,\ell)\not=(4,0)\ ,\]
    it follows that
    \[ (Q_j)_\ast A^4_{hom}(X)=0\ \ \ \hbox{for\ all\ }j\ .\]
    In view of equality (\ref{devel}), we thus have
    \[ (\Gamma^{\circ N})_\ast=0\colon\ \ \ A^4_{hom}(X)\ \to\ A^4_{hom}(X)\ .\]
   \end{proof}

For special correspondences, one can do better:

\begin{proposition} Let $X$ be the very special EPW sextic. Let $\Gamma\in A^4(X\times X)$ be a correspondence such that
  \[ \Gamma^\ast =0\colon\ \ H^{4,0}(X)\ \to\ H^{4,0}(X)\ .\]
  Assume moreover that $\Gamma$ can be written as
  \[ \Gamma={\displaystyle\sum_{i=1}^r}c_i \Gamma_{\sigma_i}\ \ \ \hbox{in}\ A^4(X\times X)\ ,\]
  with $c_i\in\QQ$ and $\sigma_i\in\hbox{Aut}(X)$
   induced by a $G$--equivariant automorphism $\sigma_i^E\colon E^4\to E^4$, where $X=E^4/(G)$ and $\sigma_i^E$ is a group homomorphism.
  Then
  \[ \Gamma^\ast=0\colon\ \ A^4_{hom}(X)\ \to\ A^4_{hom}(X)\ .\]
  \end{proposition}

\begin{proof} Let us write $A=E^4$, and $X^\prime:=A/(G^\prime)$ for the double cover of $X$ with $\dim H^{2,0}(X^\prime)=1$. The projection $g\colon X^\prime\to X$ induces an isomorphism
  \[ g^\ast\colon\ \ H^{4,0}(X)\ \xrightarrow{\cong}\ H^{4,0}(X^\prime)\ ,\]
  with inverse given by ${1\over d} g_\ast$.
 Let $\sigma^\prime_i\colon X^\prime\to X^\prime$ ($i=1,\ldots,r$) be the automorphism induced by $\sigma_i^E$.
  For each $i=1,\ldots,r$, there is a commutative diagram
  \[ \begin{array}[c]{ccc}
        H^{4,0}(X^\prime) & \xrightarrow{ (\sigma_i^\prime)^\ast}& H^{4,0}(X^\prime)\\
       {\scriptstyle g^\ast}  \uparrow \ \ \ \  && \ \ \ \ \downarrow {\scriptstyle g_\ast}\\
       H^{4,0}(X) & \xrightarrow{ (\sigma_i)^\ast}& H^{4,0}(X^)\\
       \end{array}\]
   Defining a correspondence
   \[ \Gamma^\prime=     {\displaystyle\sum_{i=1}^r}c_i \Gamma_{\sigma^\prime_i}\ \ \ \hbox{in}\ A^4(X^\prime\times X^\prime)\ ,\]    
   we thus get a commutative diagram
    \[ \begin{array}[c]{ccc}
        H^{4,0}(X^\prime) & \xrightarrow{ (\Gamma^\prime)^\ast}& H^{4,0}(X^\prime)\\
       {\scriptstyle g^\ast}  \uparrow \ \ \ \  && \ \ \ \ \downarrow {\scriptstyle g_\ast}\\
       H^{4,0}(X) & \xrightarrow{ \Gamma^\ast}& H^{4,0}(X^)\\
       \end{array}\]
      
    The assumption on $\Gamma^\ast$ thus implies that 
    \[      (\Gamma^\prime)^\ast =0\colon\ \ H^{4,0}(X^\prime)\ \to\ H^{4,0}(X^\prime)\ .\]
    Since (by construction of $X^\prime$) the cup--product map
    \[ H^{2,0}(X^\prime)\otimes H^{2,0}(X^\prime)\ \to\ H^{4,0}(X^\prime) \]
    is an isomorphism of $1$--dimensional $\C$--vector spaces, we must have that
    \[   (\Gamma^\prime)^\ast =0\colon\ \ H^{2,0}(X^\prime)\ \to\ H^{2,0}(X^\prime)\ .\]    
    It is readily seen this implies
    \begin{equation}\label{zero}  {}^t \Gamma^\prime\circ \Pi^{X^\prime}_{2,0}=0\ \ \ \hbox{in}\ H^8(X^\prime\times X^\prime)\ .\end{equation}
    
    Let $\Gamma_A$ denote the correspondence
     \[ \Gamma_A:=  {\displaystyle\sum_{i=1}^r}c_i \Gamma_{\sigma^E_i}\ \ \ \hbox{in}\ A^4(A\times A)\ .\]    
      Let $p^\prime\colon A\to X^\prime=A/(G^\prime)$ denote the quotient morphism. There are relations
    \begin{equation}\label{rel} \begin{split}   {}^t \Gamma_{\sigma^\prime} &= {1\over \vert G^\prime\vert}\ \Gamma_{p^\prime}\circ {}^t \Gamma_{A}\circ {}^t \Gamma_{p^\prime}\ \ \ \hbox{in}\ 
              A^4(X^\prime\times X^\prime)\ ,\\
                          \Pi^{X^\prime}_{2,0}&={1\over \vert G^\prime\vert}\ \ \Gamma_{p^\prime}\circ \Pi^A_{2,0}\circ {}^t \Gamma_{p^\prime}\ \ \    \hbox{in}\ 
              A^4(X^\prime\times X^\prime)\ \\
              \end{split}\end{equation}
          (the first relation is by construction of the automorphisms $\sigma_i^\prime$; the second relation can be taken as definition, cf. corollary \ref{symref}).  Plugging in these relations in equality (\ref{zero}), one obtains
        \[  \Gamma_{p^\prime}\circ {}^t \Gamma_{A}\circ {}^t \Gamma_{p^\prime}\circ \Gamma_{p^\prime}\circ \Pi^A_{2,0}\circ {}^t \Gamma_{p^\prime}=0\ \ \ \hbox{in}\ H^8(X^\prime\times X^\prime)\ .\]
       Composing with ${}^t \Gamma_{p^\prime}$ on the left and $\Gamma_{p^\prime}$ on the right, this implies in particular that
       \[  {}^t \Gamma_{p^\prime}\circ     \Gamma_{p^\prime}\circ {}^t \Gamma_{A}\circ {}^t \Gamma_{p^\prime}\circ \Gamma_{p^\prime}\circ \Pi^A_{2,0}\circ {}^t \Gamma_{p^\prime}\circ \Gamma_{p^\prime}=0\ \ \ \hbox{in}\ H^8(A\times A)\ .\]
    Using the standard relation ${}^t \Gamma_{p^\prime}\circ     \Gamma_{p^\prime}={1\over \vert G^\prime\vert}\ \sum_{g\in G^\prime} \Gamma_g$, this simplifies to
     \[  \bigl(  \sum_{g\in G^\prime}\Gamma_g\bigr)  \circ {}^t \Gamma_{A} \circ \bigl(  \sum_{g\in G^\prime}\Gamma_g\bigr)\circ \Pi^A_{2,0} =0\ \ \ \hbox{in}\ H^8(A\times A)\ .\]
     The left--hand side is a symmetrically distinguished cycle which is homologically trivial, and so it is rationally trivial (theorem \ref{os}). That is,
     \[  \bigl(  \sum_{g\in G^\prime}\Gamma_g\bigr)  \circ {}^t \Gamma_{A} \circ \bigl(  \sum_{g\in G^\prime}\Gamma_g\bigr)\circ \Pi^A_{2,0} =0\ \ \ \hbox{in}\ A^4(A\times A)\ ,\]
     in other words
     \[  {}^t \Gamma_{p^\prime}\circ     \Gamma_{p^\prime}         \circ {}^t \Gamma_{A} \circ {}^t \Gamma_{p^\prime}\circ     \Gamma_{p^\prime}  \circ \Pi^A_{2,0}=0     \ \ \ \hbox{in}\ A^4(A\times A)\ .\]
         Now we descend again to $X^\prime$ by composing some more on both sides:
    \[ \Gamma_{p^\prime}\circ {}^t \Gamma_{p^\prime}\circ     \Gamma_{p^\prime}         \circ {}^t \Gamma_{A} \circ {}^t \Gamma_{p^\prime}\circ     \Gamma_{p^\prime}  \circ \Pi^A_{2,0}\circ {}^t \Gamma_{p^\prime}=0     \ \ \ \hbox{in}\ A^4(X^\prime\times X^\prime)\ .\]
 Using the relations (\ref{rel}), this shimmers down to
  \[ ({}^t \Gamma^\prime)\circ \Pi^{X^\prime}_{2,0}=0\ \ \ \hbox{in}\ A^4(X^\prime\times X^\prime)\ .\]   
  This implies that
  \[ (\Gamma^\prime)^\ast=0\colon\ \ \ A^2_{hom}(X^\prime)\ \to\ A^2_{hom}(X^\prime)\ .\]
  Since $ A^4_{(4)}(X^\prime)$ equals the image of the intersection product $A^2_{hom}(X^\prime)\otimes A^2_{hom}(X^\prime)\to A^4(X^\prime)$ (proposition \ref{surj}), we also have that
  \[ (\Gamma^\prime)^\ast=0\colon\ \ \  A^4_{(4)}(X^\prime)\ \to\  A^4_{(4)}(X^\prime)\ .\]
  The commutative diagram
   \[ \begin{array}[c]{ccc} 
       A^4_{(4)}(X^\prime)   &\xrightarrow{(\Gamma^\prime)^\ast}& A^4_{(4)}(X^\prime) \\
       {\scriptstyle g^\ast}\uparrow\ \ \ \ &&  \ \ \ \ \uparrow  {\scriptstyle g^\ast}\\
       A^4_{hom}(X)   &\xrightarrow{\Gamma^\ast}& \ \ A^4_{hom}(X)\ , \\
       \end{array}\]
     in which vertical arrows are isomorphisms (proof of theorem \ref{main3}), now implies that
     \[ \Gamma^\ast=0\colon\ \ \ A^4_{hom}(X)\ \to\ A^4_{hom}(X)\ .\]
       \end{proof}

\subsection{Voisin conjecture}
\label{ssv}

Motivated by the Bloch--Beilinson conjectures, Voisin formulated the following conjecture:

\begin{conjecture}[Voisin \cite{V9}]\label{conjVois} Let $X$ be a smooth Calabi--Yau variety of dimension $n$. Let $a,a^\prime\in A^n_{hom}(X)$ be two $0$--cycles of degree $0$. Then
  \[ a\times a^\prime =(-1)^n a^\prime\times a\ \ \ \hbox{in}\ A^{2n}(X\times X)\ .\]
  \end{conjecture}

It seems reasonable to expect this conjecture to go through for Calabi--Yau's that are quotient varieties.
In particular, conjecture \ref{conjVois} should be true for all EPW sextics that are quotient varieties. We can prove this for the very special EPW sextic:

\begin{proposition}\label{prV} Let $X$ be the very special EPW sextic. Let $a,a^\prime\in A^4_{hom}(X)$. Then
 \[ a\times a^\prime = a^\prime\times a\ \ \ \hbox{in}\ A^8(X\times X)\ .\]
 \end{proposition}

\begin{proof} As we have seen, there is a finite morphism $p\colon A\to X$, where $A$ is an abelian fourfold and
  \[ p^\ast\colon\ \ A^4_{hom}(X)\ \to\ A^4_{(4)}(A)=(\Pi^A_4)_\ast A^4(A) \]
  is a split injection. (The inverse to $p^\ast$ is given by a multiple of $p_\ast$.)
  Proposition \ref{prV} now follows from the following fact: any $c,c^\prime\in A^4_{(4)}(A)$ verify
    \[ c\times c^\prime= c^\prime\times c\ \ \ \hbox{in}\ A^8(A\times A)\ ;\]
    this is \cite[Example 4.40]{Vo}.
    \end{proof}

\vskip1cm

\begin{nonumberingt} The ideas developed in this note grew into being during the Strasbourg 2014---2015 groupe de travail based on the monograph \cite{Vo}. Thanks to all the participants of this groupe de travail for a stimulating atmosphere. I am grateful to Bert van Geemen and to the referee for helpful comments, and to Charles Vial for making me appreciate \cite{OS}, which is an essential ingredient in this note. 

Many thanks to Yasuyo, Kai and Len for hospitably receiving me in the Schiltigheim Math. Research Institute, where this note was written.
\end{nonumberingt}

\end{document}